\documentclass[10pt]{amsart}

\title[Modular abelian surfaces of small conductor with nontrivial $\Sha$]{Modular abelian surfaces of small conductor with nontrivial Tate--Shafarevich groups}

\author{Sam Frengley} 
\address{Inria and Laboratoire d'Informatique de l'École polytechnique, CNRS, Institut Polytechnique de Paris, Palaiseau, France}
\email{samuel.frengley@inria.fr}

\author{Dylan Laird} 
\address{École Normale Supérieure de Lyon, Lyon, France}
\email{dylan.laird@ens-lyon.fr}

\date{July 23, 2026}

\usepackage{colonequals}	
\usepackage{amsmath}		
\usepackage{amssymb}		
\usepackage{amsthm}			
\usepackage{amsfonts}	
\usepackage{thmtools}
\usepackage{enumitem}	
\usepackage[pagebackref=true,hypertexnames=false]{hyperref}

\def\MR#1{\href{http://www.ams.org/mathscinet-getitem?mr=#1}{MR#1}}

\usepackage{mathrsfs}
\usepackage{caption}
\usepackage{mathtools}
\usepackage{array,booktabs}
\usepackage{color}
\usepackage{subcaption}
\usepackage{float}
\usepackage[utf8]{inputenc}
\usepackage{pdflscape}

\usepackage{cleveref}

\usepackage{tikz}
\usepackage{tikz-cd}

\usepackage{array}

\usepackage{longtable}
\usepackage{makecell}
\usepackage{adjustbox}
\usepackage{multirow}

\renewcommand{\arraystretch}{1.2}
\usepackage{chngcntr}
\counterwithin{table}{section}

\usepackage{placeins}

\DeclareFontFamily{U}{wncy}{}
\DeclareFontShape{U}{wncy}{m}{n}{<->wncyr10}{}
\DeclareSymbolFont{mcy}{U}{wncy}{m}{n}
\DeclareMathSymbol{\Sha}{\mathord}{mcy}{"58} 

\newcommand{\CC}{\mathbb{C}}
\newcommand{\FF}{\mathbb{F}}

\newcommand{\QQ}{\mathbb{Q}}

\newcommand{\ZZ}{\mathbb{Z}}
\newcommand{\cO}{\mathcal{O}}
\newcommand{\cL}{\mathcal{L}}

\newcommand{\fq}{\mathfrak{q}}
\newcommand{\fp}{\mathfrak{p}}
\newcommand{\fm}{\mathfrak{m}}

\newcommand{\overbar}[1]{\mkern 1.5mu\overline{\mkern-1.5mu#1\mkern-1.5mu}\mkern 1.5mu}

\newcommand{\LMFDBLabelMF}[1]{%
\textnormal{\texttt{\href{https://www.lmfdb.org/ModularForm/GL2/Q/holomorphic/#1}{{#1}}}}}

\DeclareMathOperator{\rank}{rk}

\DeclareMathOperator{\Frob}{Frob}
\DeclareMathOperator{\Gal}{Gal}
\DeclareMathOperator{\End}{End}

\DeclareMathOperator{\Aut}{Aut}

\DeclareMathOperator{\tr}{tr}
\DeclareMathOperator{\GL}{GL}
\DeclareMathOperator{\SL}{SL}

\DeclareMathOperator{\Vis}{Vis}
\DeclareMathOperator{\Nm}{Nm}

\newcommand{\bigp}{4000}

\usepackage{algorithm}
\usepackage{algorithmic}

\makeatletter 
 
\@addtoreset{algorithm}{section} 
\makeatother

\newtheorem{theorem}[algorithm]{Theorem}
\newtheorem{prop}[algorithm]{Proposition}
  
\newtheorem{lemma}[algorithm]{Lemma}
\newtheorem{conj}[algorithm]{Conjecture}
\theoremstyle{definition}\newtheorem{defn}[algorithm]{Definition}
\theoremstyle{definition}\newtheorem{remark}[algorithm]{Remark}
\theoremstyle{definition}
\theoremstyle{definition}
\theoremstyle{definition}\newtheorem{question}[algorithm]{Question}
\theoremstyle{definition}

\newcounter{filters}
\newcommand{\filter}[1]{%
\vspace{2mm}\refstepcounter{filters}\noindent\underline{Filter {\thefilters} (#1):}
}

\begin{document}

\begin{abstract}
We exhibit examples of geometrically simple abelian surfaces $A/\QQ$ with conductor bounded by $(10\,000)^2$ whose Tate--Shafarevich groups contain a subgroup isomorphic to $(\mathbb{Z}/p\mathbb{Z})^2$ for each $p = 5, 7, 11, 13$. To find these examples we generalise work of Cremona--Freitas to give a candidate list of all congruences of a certain type between pairs of weight $2$ newforms $f \in S_2^{\textnormal{new}}(\Gamma_0(N))$ and $g \in S_2^{\textnormal{new}}(\Gamma_0(M))$ contained in the LMFDB (i.e., with $N, M \leq 10\,000$) and with coefficient fields of degree $\leq 4$.  Passing from the modular forms to the corresponding abelian varieties we use visibility to (unconditionally) prove the existence of non-trivial elements of the Tate--Shafarevich group. Finally we construct an example of an abelian surface with $(\mathbb{Z}/7\mathbb{Z})^2 \subset \Sha(A/\mathbb{Q})$ which is (conjecturally) not visible in any abelian threefold. 
\end{abstract}

\maketitle

\vspace{-5mm}
\section{Introduction}
Let $K$ be a number field and let $G_K = \Gal(\overbar{K}/K)$ be the absolute Galois group of $K$. Let $A/K$ be an abelian variety and let $\Sha(A/K) \subset H^1(K, A)$ be the Tate--Shafarevich group of $A$, i.e., the subgroup 
\[ \Sha(A/K) =  \ker \left(H^1(K, A) \to \prod_v H^1(K_v, A) \right).\]
In this article we prove the following theorem. 

\begin{theorem}
    \label{thm:examples}
    Let $(A, p)$ be any of the pairs in Table~\ref{table:examples} where $A/\QQ$ is an abelian surface and $p$ is a prime number.
    Then $A/\QQ$ is geometrically simple and the Tate--Shafarevich group $\Sha(A/\QQ)$ contains a subgroup isomorphic to $(\ZZ/p\ZZ)^2$.
\end{theorem}

\begin{table}[ht]
  \centering
  \begin{tabular}{ccc||ccc}
    $p$  & Label of $f_A$                        & $K_A$             &      $p$ & Label of $f_A$                         & $K_A$             \\
    \hline
    $5$  & \LMFDBLabelMF{4968.2.a.bf} & $\QQ(\sqrt{41})$  &  $7$   & \LMFDBLabelMF{8475.2.a.r}  & $\QQ(\sqrt{2})$  \\
    $5$  & \LMFDBLabelMF{5056.2.a.ba} & $\QQ(\sqrt{6})$  & $7$   & \LMFDBLabelMF{8976.2.a.bm} & $\QQ(\sqrt{2})$\\
    $5$  & \LMFDBLabelMF{8784.2.a.bj} & $\QQ(\sqrt{6})$  &  $7$   & \LMFDBLabelMF{9603.2.a.o}  & $\QQ(\sqrt{2})$  \\
    $7$   & \LMFDBLabelMF{3200.2.a.bm} & $\QQ(\sqrt{2})$  &  $11$  &\LMFDBLabelMF{9025.2.a.r}  & $\QQ(\sqrt{5})$ \\
     $7$   & \LMFDBLabelMF{5220.2.a.r}  & $\QQ(\sqrt{57})$  &  $13$  & \LMFDBLabelMF{6776.2.a.r}  & $\QQ(\sqrt{3})$  \\
     $7$   & \LMFDBLabelMF{8464.2.a.bg} & $\QQ(\sqrt{2})$ &  && \\
  \end{tabular}
  \caption{Examples of geometrically simple abelian surfaces $A/\QQ$ with real multiplication by an order in the field $K_A$ and with $(\ZZ/p\ZZ)^2 \subset \Sha(A/\QQ)$. The abelian surface $A/\QQ$ is the Jacobian of the genus $2$ curve given in \Cref{tab:weier-eqns} whose isogeny class corresponds to the (Galois orbit of) weight $2$ newforms $f_A$ from the LMFDB~\cite{LMFDB}.}
  \label{table:examples}
\end{table}

To our knowledge these are the first examples of geometrically simple abelian surfaces with non-trivial $11$- and $13$-torsion in their Tate--Shafarevich groups. Works of Bruin--Flynn--Testa~\cite{BFT_DV33IOJOG2C} and Flynn~\cite{F_DV55IOJOG2C} provide examples of geometrically simple abelian surfaces $A/\QQ$ with $(\ZZ/p\ZZ)^2 \subset \Sha(A/\QQ)$ for each $p = 3, 5$ using $(p,p)$-descent. In \cite[Table~1]{AS_VEFTBSDCFMAVOAR0} Agashe and Stein used visibility in modular Jacobians to give many examples of abelian varieties (often of very large dimension) with real multiplication and nontrivial $\Sha(A/\QQ)[p]$ (including an abelian surface $A/\QQ$ with $(\ZZ/3\ZZ)^2 \subset \Sha(A/\QQ)$ in the isogeny class corresponding to the newform \LMFDBLabelMF{571.2.a.c}). In \cite[Theorem~1]{F_E7TITSGOG2J} (cf. \cite[Appendix~A]{KS_CVOSBSDFMMASOQ}) the first named author found examples of geometrically simple abelian surfaces $A/\QQ$ with $(\ZZ/7\ZZ)^2 \subset \Sha(A/\QQ)$ by leveraging visibility in an abelian threefold (in particular, the example \LMFDBLabelMF{3200.2.a.bm} in \Cref{thm:examples} is given in \cite{F_E7TITSGOG2J,KS_CVOSBSDFMMASOQ}). Our approach is closely related to \cite{AS_VEFTBSDCFMAVOAR0}, \cite[Section~5]{F_VEOO7ITTSGOAEC}, and \cite{F_E7TITSGOG2J}.

\begin{remark}
  \label{remark:galois-rep}
  In the works \cite{AS_VEFTBSDCFMAVOAR0}, \cite{BFT_DV33IOJOG2C}, \cite{F_DV55IOJOG2C}, and \cite{F_E7TITSGOG2J} the examples provided have mod $p$ Galois representations which do not surject onto $\operatorname{GSp}_{2g}(\FF_p)$ and the same is true for the examples in \Cref{thm:examples} since the abelian surfaces have real multiplication (this will play a key role). It would be interesting to construct examples with surjective mod $p$ Galois representations.
\end{remark}

Let $A/\QQ$ and $B/\QQ$ be abelian varieties and let $\Delta / \QQ$ be a finite group scheme. Suppose that there exist embeddings $\Delta \hookrightarrow A$ and $\Delta \hookrightarrow B$. Consider the abelian variety $Z = (A \times B) / \Delta$ (where $\Delta$ is embedded diagonally in $A \times B$). Writing $\iota \colon A \to Z$ for the natural morphism, we define the \emph{visible subgroup}
\[
  \Vis_Z H^1(\QQ, A) = \ker( H^1(\QQ, A) \xrightarrow{\iota_*} H^1(\QQ, Z) )
\]
and accordingly $\Vis_Z \Sha(A/\QQ) = \Sha(A/\QQ) \cap \Vis_Z H^1(\QQ, A)$. We say that an element $\xi \in \Sha(A/\QQ)$ is \emph{visible in $Z$} if $\xi \in \Vis_Z \Sha(A/\QQ)$.

If $A/\QQ$ is an abelian variety we write $\End_\QQ(A)$ for the ring of endomorphisms of $A$ defined over $\QQ$. If $\fp \subset \End_\QQ(A)$ is an ideal in a subring of $\End_\QQ(A)$ we write $A[\fp] \subset A$ for the subgroup annihilated by $\fp$.

\begin{defn}
    Let $A/\QQ$ and $B/\QQ$ be simple abelian varieties and let $\fp \subset \End_\QQ(A)$ and $\fq \subset \End_\QQ(B)$ be prime ideals. We say that $A$ and $B$ are \emph{$(\fp, \fq)$-congruent} if there exists an isomorphism $A[\fp] \cong B[\fq]$ of $G_\QQ$-modules.    

    We say that an element $\xi \in \Sha(A/\QQ)$ is \emph{made visible by the $(\fp, \fq)$-congruence $\psi \colon A[\fp] \cong B[\fq]$} if $\xi$ is visible in $Z = (A \times B) / \Delta$ where $\Delta = \{(x, \psi(x)) : x \in A[\fp]\}$ is the graph of the isomorphism $\psi$.
\end{defn}

We deduce \Cref{thm:examples} from the following more precise result.

\begin{theorem}
    \label{thm:examples-vis}
    Let $(A,B,p)$ be any of the tuples in \Cref{tab:full-examples} where $A/\QQ$ and $B/\QQ$ are abelian varieties and $p$ is a prime number. If $A$ or $B$ is marked with either a dagger $(\dagger)$ or an asterisk $(*)$ in \Cref{tab:full-examples} assume that all of its Tamagawa numbers are coprime to $p$. Then there exist prime ideals $\fp \subset \End_\QQ(A)$ and $\fq \subset \End_\QQ(B)$ such that the Tate--Shafarevich group $\Sha(A/\QQ)[\fp]$ contains a subgroup isomorphic to $(\ZZ/p\ZZ)^2$ which is made visible by a $(\fp,\fq)$-congruence $A[\fp] \cong B[\fq]$.
\end{theorem}

\begin{remark}
  In those cases of \Cref{thm:examples-vis} marked with a $(\dagger)$ the Tamagawa numbers of $A$ and $B$ are coprime to $p$ according to calculations using the unpublished work~\cite{DJ_RMOC}.
\end{remark}

For any abelian variety $A/\QQ$ and each element $\xi \in \Sha(A/\QQ)$ there exists \emph{some} abelian variety $Z/\QQ$ such that $\xi \in \Vis_Z \Sha(A/\QQ)$ (see \cite[p.5 Proposition]{CM_VEITSTG} or \cite[Proposition~2.3]{AS_VOSTGOAV}). It is therefore natural to ask for the \emph{visibility dimension} of $\xi \in \Sha(A/\QQ)$ which is defined to be the minimum dimension of an abelian variety $Z/\QQ$ visualising $\xi$.

To this end, we draw the reader's attention to the abelian varieties with labels \LMFDBLabelMF{9603.2.a.o}, \LMFDBLabelMF{6962.2.a.q} and \LMFDBLabelMF{9025.2.a.r} in \Cref{tab:full-examples}. In each of these cases $A/\QQ$ is an abelian surface and the subgroup $(\ZZ/p\ZZ)^2 \subset \Sha(A/\QQ)$ is made visible by a congruence with an abelian surface $B/\QQ$. 

\begin{conj}
    \label{conj:invisible-7}
    The Tate--Shafarevich group of the abelian surface $A/\QQ$ with label \LMFDBLabelMF{9603.2.a.o} contains an element of order $7$ which is not visible in any abelian threefold, in particular its visibility dimension is $4$.
\end{conj}

\Cref{conj:invisible-7} follows if one can determine the rational points on a pair of explicit plane quartic curves (given in \Cref{sec:invis}).

\begin{question}
    \label{qn:invisible}
    Can the elements of order $11$ in the Tate--Shafarevich groups of the abelian surfaces with labels \LMFDBLabelMF{6962.2.a.q} and \LMFDBLabelMF{9025.2.a.r} be visualised in an abelian threefold?
\end{question}

\subsection{Congruences of weight \texorpdfstring{$2$}{2} newforms}
\label{sec:congs-wt-2}

The key technical input in this article is to find several examples of $(\fp,\fq)$-congruent abelian varieties over $\QQ$. Testing each of these examples against the criterion in \cite[Theorem~2.2]{F_VEOO7ITTSGOAEC} (see \Cref{thm:vis-crit}) yields \Cref{thm:examples,thm:examples-vis}. To construct examples of congruences, we exploit the correspondence between abelian varieties with real multiplication and weight $2$ newforms. We then utilise the LMFDB~\cite{LMFDB} to find congruences of newforms. 

Let $\mathcal{O}$ be an order in a totally real number field $K$ of degree $d$. By work of Shimura~\cite{S_OTFOTJVOAMFF} we may associate to a Galois orbit of newforms $f \in S_2^{\textnormal{new}}(\Gamma_0(N))$ with coefficient ring $\mathcal{O}$ (and trivial nebuntypus) an isogeny class of abelian varieties $A/\QQ$ of dimension $d$, conductor $N^d$, and with $\End_\QQ(A) \cong \cO$. 

Let $A/\QQ$ be such an abelian variety associated to a Galois orbit of newforms $f \in S_2^{\textnormal{new}}(\Gamma_0(N))$. Let $p$ be a prime number coprime to $N$ and unramified in $\cO$. Following \cite[Section~2]{R_GAODPOAVWRM} we have 
\[V_p(A) \cong \bigoplus_\fp V_\fp(A)\]
where $V_p(A) = T_p(A) \otimes_{\ZZ_p} \QQ_p$ is the $p$-adic Tate module, $V_\fp(A) = V_p(A) \otimes_{K \otimes \QQ_p} K_\fp$, and the direct sum ranges over those prime ideals $\fp$ of $K$ lying above $p$. 
In particular, by compactness and continuity, after choosing an $\cO_\fp$-lattice spanning $V_\fp(A)$ we obtain Galois representations $\rho_{f,\fp} \colon G_{\QQ} \to \GL_2(\mathcal{O}_\fp)$. 
For each prime $\ell \nmid N p$ the Galois representation $\rho_{f,\fp}$ is unramified at $\ell$ and we have $a_\ell(f) = \tr \rho_{f,\fp}(\Frob_\ell)$. 

For a prime $\fp$ of $\mathcal{O}$ we write $\FF_\fp = \mathcal{O}/\fp$ for the residue field at $\fp$ and let $\bar{\rho}_{f,\fp} \colon G_\QQ \to \GL_2(\FF_\fp)$ be the representation obtained by reducing $\rho_{f,\fp}$ modulo $\fp$. The semisimplification of the Galois representation $\bar{\rho}_{f,\fp}$ does not depend on the choice of $\cO_\fp$-lattice and is isomorphic to the semisimplification of the Galois representation $\bar{\rho}_{A,\fp}$ obtained from the action of $G_\QQ$ on $A[\fp]$. 

For a newform $f \in S_2^{\textnormal{new}}(\Gamma_0(N))$ we write $\cO_f$ for the coefficient ring of $f$ and we write $K_f = \cO_f \otimes \QQ$.

\begin{defn}
  Let $f \in S_2^{\textnormal{new}}(\Gamma_0(N))$ and $g \in S_2^{\textnormal{new}}(\Gamma_0(M))$ be newforms. We say that $f$ and $g$ are \emph{$p$-congruent} if there exists a field $L/\QQ$, a prime ideal $\fm \subset \cO_L$ dividing $p$, and embeddings $K_f \hookrightarrow L$ and $K_g \hookrightarrow L$ such that $a_\ell(f) - a_\ell(g) \in \fm$ for all $\ell \nmid pNM$.

  We will further say that $f$ and $g$ are \emph{$p$-congruent modulo an unramified prime} if $L$ and $\fm$ may be chosen so that the ramification index $e(\fm/p) = 1$.
\end{defn} 

It is clear that if $f$ and $g$ are $p$-congruent, then the same is true for all elements of their Galois orbits. 
Moreover, if $A/\QQ$ and $B/\QQ$ are $(\fp,\fq)$-congruent then the corresponding (Galois orbits of) modular forms $f$ and $g$ are $p$-congruent (as is well known, the converse is true up to semisimplification, see \Cref{prop:sturm-cong}\ref{iii:ss-isom}). This suggests that, in order to search for $(\fp, \fq)$-congruences between abelian varieties, one may instead search for $p$-congruences of modular forms. Our main tool for finding the examples in \Cref{thm:examples} is the following theorem which builds on work of Cremona--Freitas~\cite{CF_GMFTSTOCBEC}, in which the coefficient fields are $\QQ$.

\Cref{thm:congruence-classification} lists all potential $p$-congruences of newforms in the LMFDB, for primes in the range $5 \leq p \leq 4000$, and levels up to $10\,000$.

\begin{theorem}
    \label{thm:congruence-classification}
    Suppose that $f \in S_2^{\textnormal{new}}(\Gamma_0(N))$ and $g \in S_2^{\textnormal{new}}(\Gamma_0(M))$ are newforms of level $N,M \leq 10\,000$ which are not Galois conjugate. Suppose that the coefficient fields $K_f$ and $K_g$ satisfy $\deg(K_f), \deg(K_g) \leq 4$ and that $\cO_f$ and $\cO_g$ are maximal orders. If there exists a prime number $5 \leq p \leq \bigp$ such that $f$ and $g$ are $p$-congruent modulo an unramified prime, then the triple $(f, g, p)$ is contained in the corresponding file \textnormal{\texttt{data/congruences/p.txt}} in \cite{ELECTRONIC}. 
    
    Moreover, if $23 \leq p \leq \bigp$ and $\deg(K_f), \deg(K_g) \leq 2$ then $(f,g,p)$ is listed in \Cref{tab:mf-congs} and the converse holds (i.e., $f$ and $g$ are $p$-congruent modulo an unramified prime).
\end{theorem}

\begin{remark} \hfill
  \begin{enumerate}
  \item
    It is almost certain that each triple in \Cref{thm:congruence-classification} is an example of a $p$-congruence. Indeed for each such pair the traces of Frobenius $a_{\ell}(f)$ and $a_{\ell}(g)$ are congruent for all $\ell < 1000$ with $\ell \nmid pNM$. To prove this in any individual case, one may compute Fourier coefficients up to a related Sturm bound (see \Cref{prop:sturm-cong}\ref{iii:equalities}).
  \item
    For each $d \leq 4$ we found a $p$-congruence of the form in \Cref{thm:congruence-classification} with $\deg(K_f), \deg(K_g) \leq d$ and $p = p_d$ where $p_1 = 13$, $p_2 = 59$, $p_3 = 197$, and $p_4 = 739$ (and these were the largest such primes). The corresponding pairs of congruent forms are (\LMFDBLabelMF{735.2.a.d}, \LMFDBLabelMF{9555.2.a.h}), (\LMFDBLabelMF{4332.2.a.e}, \LMFDBLabelMF{4332.2.a.f}), (\LMFDBLabelMF{7854.2.a.bb}, \LMFDBLabelMF{7854.2.a.bc}), and (\LMFDBLabelMF{8775.2.a.bg}, \LMFDBLabelMF{8775.2.a.bo}) respectively. It is likely that there exist no $p$-congruences of this form (in particular with $N,M \leq 10\,000$) for any $p > p_d$. It may be possible to prove this by adapting the methods used in the proof of \cite[Theorem~1.3]{CF_GMFTSTOCBEC}.
  \item
    It is notable that for each of the congruences visualising an element of $\Sha(A/\QQ)[p]$ in \Cref{thm:examples-vis} the prime $p$ splits in $\End_\QQ(A)$.
  \end{enumerate}    
\end{remark}

\subsection{Outline of the paper}
In \Cref{sec:prove-congs-ss} we generalise a result of Kraus--Oesterlé \cite[Proposition~4]{KO_SUQDBM} which allows us to prove $p$-congruences of modular forms by comparing Fourier coefficients up to a Sturm bound (this result is likely known to experts). Applying this result to the pairs of forms in \Cref{tab:full-examples} (which includes the examples in \Cref{thm:examples}) we prove \Cref{thm:congruence-classification-av}, which gives several examples of $(\fp,\fq)$-congruent pairs of abelian varieties of dimension $\leq 2$.

In \Cref{sec:sieve} we describe our generalisation of the algorithm given by Cremona--Freitas~\cite{CF_GMFTSTOCBEC} for finding examples of $p$-congruences of newforms and apply this to prove \Cref{thm:congruence-classification}.

Note that if one is only interested in \emph{proving} \Cref{thm:examples} \Cref{sec:sieve} is not required, however it is \Cref{thm:congruence-classification} which provides the candidates which we search through to find these examples.
With this in mind, in \Cref{sec:filter} we describe how we filter for abelian varieties whose Tate--Shafarevich groups may be nontrivial, starting from \Cref{thm:congruence-classification}. This step relies crucially on the works \cite{CEHJMPV_ADOCWMJ,CEHJMPV_ADOCWMJ_electronic} who give explicit abelian surfaces whose associated newforms are contained in the LMFDB. At the end of \Cref{sec:filter} we prove \Cref{thm:examples} by combining \Cref{thm:congruence-classification-av} with local criteria due to Agashe--Stein~\cite[Theorem~3.1]{AS_VOSTGOAV} and Fisher~\cite[Theorem~2.2]{F_VEOO7ITTSGOAEC} (see also~\Cref{thm:vis-crit} and \Cref{sec:theorem-agashe-stein}).

We study closely the case of the abelian surface \LMFDBLabelMF{9603.2.a.o} in \Cref{sec:invis} and in particular we provide evidence for \Cref{conj:invisible-7}.

The code associated to this article is written in \textsc{Magma}~\cite{MAGMA} and \textsc{Sagemath}~\cite{sagemath} and is available through the GitHub repository \cite{ELECTRONIC}. The larger datafiles are available through Zenodo at \cite{DATA}.

\subsection{Acknowledgements}
We thank Martin Azon, Raymond van~Bommel, Tom Fisher, and Céline Maistret for helpful comments and discussions. We are especially grateful to Brendan Creutz for discussions about \Cref{sec:invis}, to Tim Dokchitser for computing several Tamagawa numbers using~\cite{DJ_RMOC}, to Kimball Martin for sharing with us the RM abelian surfaces corresponding to newforms in the LMFDB from~\cite{CEHJMPV_ADOCWMJ,CEHJMPV_ADOCWMJ_electronic}, and to several anonymous referees whose feedback greatly improved this article. Much of the work for this article was done while the authors were based at the University of Bristol. \textsc{ChatGPT-5.6 Sol} was used for copy-editing this article. SF was supported in part by the Royal Society through Céline Maistret's Dorothy Hodgkin Fellowship and also by the HYPERFORM consortium, funded by France through Bpifrance.

\section{Proving congruences of modular forms}
\label{sec:prove-congs-ss}

In this section we give a computational criterion (\Cref{prop:sturm-cong}) for proving whether a pair of forms are $p$-congruent, adapting (and closely following) \cite[Proposition~4]{KO_SUQDBM} (which is the case of \Cref{prop:sturm-cong} where $\cO_f = \cO_g = \ZZ$).



\begin{lemma}
\label{lemma:mod-p-is-enough}
    Let $f \in S_2^{\textnormal{new}}(\Gamma_0(N))$ and $g \in S_2^{\textnormal{new}}(\Gamma_0(M))$ be newforms whose coefficient rings $\cO_f$ and $\cO_g$ are the rings of integers of $K_f$ and $K_g$ respectively. Then $f$ and $g$ are $p$-congruent modulo an unramified prime if and only if there exist unramified prime ideals $\fp \subset \cO_f$ and $\fq \subset \cO_g$ lying above $p$, and an isomorphism $\FF_\fp \cong \FF_\fq$ such that $a_\ell(f) \equiv a_\ell(g)$ in $\FF_\fp$ for all $\ell \nmid pNM$.
\end{lemma}

\begin{proof}
  First suppose that the Galois orbits of $f$ and $g$ are $p$-congruent. By definition there exists a field $L$, a prime ideal $\fm \subset \cO_L$, and embeddings $K_f \hookrightarrow L$ and $K_g \hookrightarrow L$ so that $a_\ell(f) \equiv a_\ell(g) \pmod \fm$ for all $\ell \nmid pNM$. Let $\fp = \cO_f \cap \fm$ and $\fq = \cO_g \cap \fm$. Since $\cO_f$ is generated by those $a_\ell(f)$ with $\ell \nmid pNM$ (and respectively for $g$) the fields $\FF_{\fp}$ and $\FF_{\fq}$ have the same cardinality. In particular, they are the same subfields of $\FF_{\fm}$ and this gives the isomorphism $\FF_{\fp} \cong \FF_{\fq}$ from the statement.

  Conversely suppose that there exists an isomorphism $\tau \colon \FF_\fp \cong \FF_\fq$ such that $a_\ell(f) \equiv a_\ell(g)$ in $\FF_\fp$ for all $\ell \nmid pNM$. Choose embeddings $K_f, K_g \hookrightarrow \overbar{\QQ}$ and let $\widetilde{L}$ be the Galois closure of the compositum of $K_f$ and $K_g$ inside $\overbar{\QQ}$. Let $\cO_{\widetilde L} \subset \widetilde{L}$ be the ring of integers of $\widetilde{L}$. After possibly changing the embedding $K_g \hookrightarrow \widetilde{L}$ we may choose a prime $\widetilde{\fm} \subset \cO_{\widetilde L}$ lying above both $\fp$ and $\fq$. Let $I \subset D \subset \Gal(\widetilde{L}/\QQ)$ be the decomposition and inertia groups of ${\widetilde{\fm}}$ and choose $\widetilde{\tau} \in D$ so that the image of $\widetilde{\tau}$ in $D / I = \Gal(\FF_{\widetilde{\fm}}/\FF_p)$ restricts to $\tau$ on $\FF_\fp$. Let $L \subset \widetilde{L}$ be the field fixed by $I$ and let $\fm = L \cap \widetilde{\fm}$. Then $e(\fm/p) = 1$ and by construction for all $\ell \nmid pNM$ we have $\widetilde{\tau}(a_\ell(f)) \equiv a_\ell(g) \pmod{\fm}$, as required.
\end{proof}

For a positive integer $\eta$ we write $\mu(\eta) = \eta \cdot \prod_{p \mid \eta} (1+1/p)$ for the index of $\Gamma_0(\eta)$ in $\SL_2(\ZZ)$. For a prime $\ell$ we write $v_\ell$ for the $\ell$-adic valuation, normalised so that $v_\ell(\ell) = 1$. Given a representation $\rho$ we write $\rho^{\mathrm{ss}}$ for the semisimplification of $\rho$.

\begin{prop}
  \label{prop:sturm-cong}
  Let $f \in S_2^{\textnormal{new}}(\Gamma_0(N))$ and $g \in S_2^{\textnormal{new}}(\Gamma_0(M))$ be newforms whose coefficient rings $\cO_f$ and $\cO_g$ are the rings of integers of $K_f$ and $K_g$ respectively. Let $S$ be the set of prime numbers $\ell$ such that $v_{\ell}(N) = v_{\ell}(M) = 1$ and $a_{\ell}(f) \neq a_{\ell}(g)$. Define
  \begin{equation*}
    \eta = \operatorname{lcm}(N,M) \prod_{\ell \in S} \ell .
  \end{equation*}
%
  The following are equivalent:
  \begin{enumerate}[label=(\arabic*)]
  \item \label{iii:fp-fq-cong-i}
    The modular forms $f$ and $g$ are $p$-congruent modulo an unramified prime.
  \item \label{iii:ss-isom}
    There exist unramified primes $\fp \subset \cO_f$ and $\fq \subset \cO_g$ and an isomorphism $\FF_\fp \cong \FF_\fq$ such that the semisimplifications $\bar{\rho}_{f, \fp}^{\mathrm{ss}}$ and $\bar{\rho}_{g, \fq}^{\mathrm{ss}}$ are isomorphic. 
  \item \label{iii:equalities}
     There exist unramified primes $\fp \subset \cO_f$ and $\fq \subset \cO_g$ and an isomorphism $\FF_\fp \cong \FF_\fq$ such that both of the following hold:
    \begin{enumerate}
    \item
      for all primes $\ell \leq \mu(\eta) / 6$ with $\ell \nmid N M$ we have $a_\ell(f) \equiv a_\ell(g)$ in $\FF_\fp$, and
    \item
      for all primes $\ell \leq \mu(\eta) / 6$ with $v_\ell(N M) = 1$ we have $a_\ell(f) a_\ell(g) \equiv \ell + 1$ in $\FF_\fp$.
    \end{enumerate}
  \end{enumerate}
\end{prop}

\begin{remark}
    Note that it may not be the case that $\eta = \operatorname{lcm}(N,M)$, even when $N = M$ (which is the case in all the examples in \Cref{table:examples,tab:mf-congs,tab:full-examples}). This is because the Sturm bound (used in the proof of \Cref{prop:sturm-cong}) implies that a weight $2$ modular form $h$ of level $\Gamma_0(\eta))$ satisfies $a_\ell(h) \equiv 0 \pmod{\fm}$ if $a_\ell(h) \equiv 0 \pmod{\fm}$ \emph{for all} $\ell \leq \mu(N)/6$ (not only those $\ell \nmid N$).
\end{remark}

\begin{proof}
  That \ref{iii:fp-fq-cong-i}$\Rightarrow$\ref{iii:ss-isom} follows immediately by combining the Brauer--Nesbitt and Chebotarev density theorems. The implication \ref{iii:ss-isom}$\Rightarrow$\ref{iii:equalities} follows from the same argument as Kraus--Oesterlé~\cite[Proposition 4]{KO_SUQDBM}, noting that by \cite[Theorem~4.6.17(2)]{M_MF} if $v_\ell(N) = 1$ then $a_\ell(f) \in \{ \pm 1\}$.
  
  For the implication \ref{iii:equalities}$\Rightarrow$\ref{iii:fp-fq-cong-i}, we again follow the argument of Kraus--Oesterlé \cite[Proposition 4]{KO_SUQDBM}. Possibly after replacing them by Galois conjugates, we may assume that $f$ and $g$ are newforms with coefficients contained in a common maximal order $\cO_L$ and that the isomorphism $\FF_\fp \cong \FF_\fq$ is induced by reduction modulo a prime $\fm \subset \cO_L$.
  
  For positive integers $m,d$ the map $\tau \mapsto d \tau$ on the complex upper half plane induces a homomorphism $R_d \colon S_2(\Gamma_0(m)) \to S_2(\Gamma_0(m d))$. Moreover, if $h \in S_2(\Gamma_0(m))$ is a cusp form then we have $L(R_d(h), s) = d^{-s} L(h, s)$. Recall that if $h \in S_2^{\textnormal{new}}(\Gamma_0(m))$ is a newform of level $m$ we have $L(h, s) = \prod_{\ell} L_\ell(h, \ell^{-s})^{-1}$ where the local factors are determined by the formula: 
  \[
     L_\ell(h, T) =
     \begin{cases}
       1- a_\ell(h) T + \ell T^2 &\textnormal{if $v_\ell(m) = 0$,}            \\
       1 - a_\ell(h) T &\textnormal{if $v_\ell(m) = 1$, and} \\
       1                 & \textnormal{if $v_\ell(m) \geq 2$.}
     \end{cases}
  \]
  For each positive integer $m$ and prime $\ell$ define linear operators
  \[
    \alpha_\ell \colon S_2(\Gamma_0(m)) \to S_2(\Gamma_0(m \ell^{v_\ell(\eta/N)})) \quad \text{and} \quad \beta_\ell \colon S_2(\Gamma_0(m)) \to S_2(\Gamma_0(m\ell^{v_\ell(\eta/M)}))
  \]
  by the case distinctions in \Cref{tab:ops}.
  Define $\alpha = \prod_\ell \alpha_\ell$ and $\beta = \prod_\ell \beta_\ell$ and write $f' = \alpha(f)$ and $g' = \beta(g)$. By construction $f', g' \in S_2(\Gamma_0(\eta))$.

  \begingroup
  \renewcommand{\arraystretch}{1.5}
  \begin{table}[t]
    \centering
    \begin{tabular}{c|c|c|c}
      $v_\ell(N)$ & $v_\ell(M)$ & $\alpha_\ell$                                & $\beta_\ell$                                \\
      \hline
      $0$         & $0$         & $1$                                     & $1$                                      \\
      $0$         & $1$         & $1 - (a_\ell(f) - a_\ell(g) )R_\ell$                                   & $1$ \\
      $1$         & $1$         & $\begin{cases} 1 & \ell \not\in S, \\ 1 - a_\ell(f) R_\ell & \ell \in S.  \end{cases}$                  & $\begin{cases} 1 & \ell \not\in S, \\ 1 - a_\ell(g) R_\ell & \ell \in S.  \end{cases}$                   \\
      $0$         & $\geq 2$    & $1 - a_\ell(f) R_\ell + \ell R_\ell^2$  & $1$                                      \\
      $1$         & $\geq 2$    & $1 - a_\ell(f) R_\ell$                  & $1$                                      \\
      $\geq 2$    & $\geq 2$    & $1$                                     & $1$                                      \\
    \end{tabular}
    \caption{Definition of the operators $\alpha_\ell$ and $\beta_\ell$ when $v_\ell(N) \leq v_\ell(M)$, the cases when $v_\ell(N) > v_\ell(M)$ are defined symmetrically.}
    \label{tab:ops}
  \end{table}
  \endgroup

  By \cite[Theorem~4.6.17(2)]{M_MF} if $v_\ell(N) = 1$ (resp. $v_\ell(M) = 1$) then we have $a_\ell(f) \in \{ \pm 1\}$ (resp. $a_\ell(g) \in \{\pm 1\}$). A direct calculation shows that the local factors $L_\ell(f', T)$ and $L_\ell(g', T)$ are equal modulo $\fm$ for every prime $\ell \leq \mu(\eta)/6$ (for details see the proof of \cite[Proposition~4]{KO_SUQDBM}). In particular we have $a_n(f') \equiv a_n(g') \pmod{\fm}$ for all integers $n \leq \mu(\eta)/6$.

  By the Sturm bound \cite[Corollary~9.19]{S_MFACA} all the Fourier coefficients of the modular forms $f'$ and $g'$ are equal modulo $\fm$. But for all $\ell \nmid NM$ we have $a_\ell(f) = a_\ell(f')$ and $a_\ell(g) = a_\ell(g')$, and the claim follows.
\end{proof}


With our target application of visibility in mind, we now turn to the task of relating congruences of modular forms to their associated abelian varieties. 

\begin{prop}
    \label{thm:congruence-classification-av}
    Let $(A, B, p)$ be any of the triples listed in \Cref{tab:full-examples} where $A/\QQ$ and $B/\QQ$ are abelian varieties and $p$ is a prime number. Then:
    \begin{enumerate}[label=(\roman*)]
        \item \label{iii:rm-claim} $A$ and $B$ have real multiplication by the rings of integers $\cO_A$ and $\cO_B$ of totally real fields of degree $\dim(A)$ and $\dim(B)$ respectively, and
        \item \label{iii:cong-claim} there exist primes $\fp \subset \cO_A$ and $\fq \subset \cO_B$ dividing $p$ such that $A$ and $B$ are $(\fp, \fq)$-congruent and $A[\fp]$ and $B[\fq]$ are irreducible $G_\QQ$-modules.
    \end{enumerate}
\end{prop}

If $J/\QQ$ is one of the abelian varieties $A/\QQ$ or $B/\QQ$ from \Cref{thm:congruence-classification-av}, then either $J$ is an elliptic curve or is given explicitly as either an elliptic curve or the Jacobian of genus $2$ curve $C/\QQ$ given in \Cref{tab:weier-eqns}. 
  

  \begin{lemma}
  \label{lemma:have-RM}
  Let $J$ be the Jacobian of any of the genus 2 curves $C/\QQ$ given in \Cref{tab:weier-eqns}. Then $J$ has real multiplication (defined over $\QQ$) by the maximal order $\cO_f$ in the corresponding field $K_f$ listed in \Cref{tab:weier-eqns}.
  \end{lemma}
  \begin{proof}
      Let $D$ be the discriminant of $K_f$. In the file \texttt{verification/2-4-prop.m} in \cite{ELECTRONIC} we exhibit a rational point on the Hilbert modular surface $Y_{-}(D)$ in the model given by Elkies--Kumar~\cite{EK_K3SAEFHMS} which corresponds to $J$. This implies that there exists a twist $J' / \QQ$ of $J$ together with an embedding $\cO_f \hookrightarrow \End_\QQ(J')$. We check with \textsc{Magma} that $\Aut_{\overbar{\QQ}}(C) = \{ \pm 1 \}$, so that by the Torelli theorem $J'$ is a quadratic twist of $J$. Let $t \colon J \to J'$ be an isomorphism over a quadratic field $L/\QQ$ and let $\chi \colon G_\QQ \to \{ \pm 1\}$ be the quadratic character associated to $L$. Then for each $\sigma \in G_\QQ$ we have $t \sigma = \chi(\sigma) t$. Then for each $\psi' \in \End_\QQ(J')$ the endomorphism $\psi = t^{-1} \psi' t \in \End_{\overbar{\QQ}}(J)$ is defined over $\QQ$, since for each $\sigma \in G_\QQ$ we have $\psi \sigma = \chi(\sigma)^2 \sigma t^{-1} \psi' t = \sigma \psi$. 
\end{proof}  

\begin{remark}
  It may instead be possible to rigorously check the existence of the real multiplication on the abelian surfaces listed in \Cref{tab:weier-eqns} using \cite{CMSV_RCOTEROAJ}.
\end{remark}

Let $J/\QQ$ be the Jacobian of any of the genus $2$ curves $C/\QQ$ in \Cref{tab:weier-eqns}. By \Cref{lemma:have-RM} the abelian variety $J$ has real multiplication defined over $\QQ$. Thus, by work of Ribet~\cite[Theorem~4.4]{R_AVOQAMF} and Serre's conjecture~\cite{KW_SMCI,KW_SMCII}, the isogeny class of $J/\QQ$ has an associated (Galois orbit of) weight $2$ newforms $f_J \in S_2^{\textnormal{new}}(\Gamma_0(N))$. We now show that the corresponding modular form is correctly recorded in \Cref{tab:full-examples}.
 
\begin{lemma}
\label{lemma:correct-eqns-forms}
  Let $(f, C)$ be any of the pairs from \Cref{tab:weier-eqns} consisting of (the LMFDB label of) a newform $f \in S_2(\Gamma_0(N))$ and a genus $2$ curve $C/\QQ$. If the pair $(f, C)$ is not marked by a $(\star)$ in \Cref{tab:weier-eqns} (i.e., if $f$ is not \LMFDBLabelMF{6962.2.a.o}, \LMFDBLabelMF{6962.2.a.q}, or \LMFDBLabelMF{9510.2.a.x}) then $f = f_J$. The same is true for the examples \LMFDBLabelMF{6962.2.a.o}, \LMFDBLabelMF{6962.2.a.q}, or \LMFDBLabelMF{9510.2.a.x} if and only if the conductor exponent of $J$ at $2$ is equal to $2$.
\end{lemma}
  \begin{proof}
    Using work of Dokchitser--Doris~\cite{DD_3TACOG2C,DD_electronic} we verify\footnote{
  At the time of writing (commit \texttt{bc73269}) there is a bug \url{https://github.com/cjdoris/Genus2Conductor/issues/6} in \cite{DD_electronic} which can cause incorrect values of the tame conductor exponent to be returned. 
  Upon setting \texttt{UseRegularModels:=false} this bug is avoided.
  However, this means that we are unable to compute the conductor rigorously in the cases of \Cref{tab:weier-eqns} labelled with a $(\star)$. 
  We are grateful to anonymous referees for bringing this to our attention.
  } 
  that the conductor of $J$ is equal to $N^2$ where $N$ is the level of $f$ (this is where we use the assumption on the conductor exponent if $(f,C)$ is marked with a $(\star)$). 

  The set of (Galois orbits of) weight $2$ newforms $S_2^{\textnormal{new}}(\Gamma_0(N))$ with coefficient field $K_f$ is listed in the LMFDB~\cite{LMFDB,BBBCCDLLRSV_CCMF} together with the traces $\operatorname{Tr}_{K_f/\QQ}(a_\ell)$ and norms $\Nm_{K_f/\QQ}(a_\ell)$ for those primes $\ell < 1000$. By \cite[Section~2.1]{FLSSSW_EEFTBSDCFMJOG2C} or \cite[Lemma~3]{MS_COG2WGRAF2WARWP} for each good prime $\ell$ of $C$ we have
  \begin{align*}
    \operatorname{Tr}_{K_f/\QQ} (a_\ell) =&\, \ell + 1 - n_1 \\
    \Nm_{K_f/\QQ} (a_\ell) =&\, (n_1^2 + n_2)/2 - (\ell + 1)n_1 - \ell 
  \end{align*}
  where $n_1 = \#C(\FF_\ell)$ and $n_2 = \#C(\FF_{\ell^2})$. In each case, computing the values $n_1$ and $n_2$ for a few small primes $\ell \nmid N$ suffices to show that $f$ is the unique (Galois orbit of) newforms whose trace and norm match that of $f_J$. Thus $f = f_J$.
    \end{proof}

\begin{proof}[Proof of \Cref{thm:congruence-classification-av}]
  Suppose that neither $A$ nor $B$ is one of the abelian surfaces labelled with a $(\star)$ in \Cref{tab:weier-eqns}. Let $f$ and $g$ be the newforms corresponding, by \Cref{lemma:correct-eqns-forms}, to $A$ and $B$ respectively. Computing Fourier coefficients in \textsc{Magma} \cite{MAGMA} up to the Sturm bound from \Cref{prop:sturm-cong}\ref{iii:equalities} we see that $f$ and $g$ are $p$-congruent modulo an unramified prime. In particular, by \Cref{prop:sturm-cong}\ref{iii:ss-isom} there exist unramified prime ideals $\fp \subset \cO_f$ and $\fq \subset \cO_g$ such that we have isomorphisms of semisimplified Galois representations $\bar{\rho}_{f, \fp}^{\mathrm{ss}} \cong \bar{\rho}_{g, \fq}^{\mathrm{ss}}$.

  However, in every case the Galois representation $\bar{\rho}_{f, \fp}$ is readily seen to be irreducible by exhibiting a single prime $\ell \nmid p N$ for which the characteristic polynomial of Frobenius $X^2 - a_{\ell}(f) X + \ell$ is irreducible modulo $\fp$. In particular $\bar{\rho}_{f, \fp} \cong \bar{\rho}_{g, \fq}$, as required.

  In the cases where a corresponding entry in \Cref{tab:weier-eqns} is marked with a $(\star)$ we use the direct method from \cite[Section~6]{F_VEOO7ITTSGOAEC} and \cite[Propositions~20 and 22]{F_E7TITSGOG2J} to show that $A[\fp] \cong B[\fq]$ as $G_\QQ$-modules.
\end{proof}

\section{Sieving the LMFDB for congruences between modular forms}
\label{sec:sieve}
We adapt the method that was used in \cite{CF_GMFTSTOCBEC} to determine $p$-congruences between elliptic curves for each $p \geq 7$. Let $\mathcal{S}$ denote the set of all Galois orbits of weight $2$ newforms $f \in S_2^{\textnormal{new}}(\Gamma_0(N))$ of conductor $ 1 \leq N \leq 10\,000$ and whose coefficient fields have degree $\leq 4$. This algorithm rests on the heuristic that there are very few pairs of forms $f, g \in \mathcal{S}$ which are not congruent, but which nevertheless have congruent $a_{\ell}$-values for several primes $\ell$. 

Fix a prime number $p$ and an algebraic closure $\overbar{\FF}_p$ of $\FF_p$. For a finite set $\mathcal{L}$ of prime numbers let $\mathcal{S}_\cL$ be the set of forms $f \in \mathcal{S}$ of level coprime to $\ell$ for each $\ell \in \cL$. For each form $f \in \mathcal{S}_\cL$ and prime $\fp \subset \cO_f$ above $p$ define the tuple
\begin{equation}
  \mathsf{h}_{\mathcal{L}}(f ; \fp) = \Big( \overbar{a_\ell}(f) \Big)_{\ell \in \mathcal{L}} \in \FF_\fp^{|\mathcal{L}|}
\end{equation}
where we write $\overbar{a_\ell}(f)$ for the image of $a_\ell(f)$ in $\FF_\fp$. The following lemma is immediate.

\begin{lemma}
  \label{lemma:hashes}
  If $f, g \in \mathcal{S}_\cL$ are $p$-congruent modulo an unramified prime then there exist primes $\fp \subset \cO_f$ and $\fq \subset \cO_g$ and a pair of embeddings $\sigma \colon \FF_\fp \hookrightarrow \overbar{\FF}_{p}$ and $\tau \colon \FF_\fq \hookrightarrow \overbar{\FF}_{p}$ such that $\sigma(\mathsf{h}_{\mathcal{L}}(f ; \fp)) = \tau(\mathsf{h}_{\mathcal{L}}(g ; \fq))$.  
\end{lemma}

Our approach is to build a hash table associating forms $f$ to the hash values $\sigma(\mathsf{h}_{\mathcal{L}}(f ; \fp))$. The following lemma provides a list of sets $\cL$ which we need to perform this calculation for.

\begin{lemma}
  \label{lemma:L-cover-S}
  Consider the sets $\cL_0, \dots, \cL_5$, each consisting of 15 prime numbers, given in \Cref{tab:Li}. 
  Then for each pair $f, g \in \mathcal{S}$ there exists $i \in \{0, \dots, 5\}$ such that $f,g \in \mathcal{S}_{\cL_i}$.
\end{lemma}

\begin{proof}
  By a brute force calculation in \textsc{Python} we check that for each pair of integers $(N,M) \in \{2, \dots, 10\,000\}^2$ there exists $0 \leq i \leq 5$ such that both $N$ and $M$ are coprime to all $\ell \in \cL_i$. The claim follows since the levels of forms $f \in \mathcal{S}$ are bounded by $10\,000$.
\end{proof}

Equipped with \Cref{lemma:hashes,lemma:L-cover-S} we are now able to prove \Cref{thm:congruence-classification}.

\begin{proof}[Proof of \Cref{thm:congruence-classification}]
  For each of the sets $\cL_i$ in \Cref{lemma:L-cover-S} and $5 \leq p \leq \bigp$ we compute a hash table which associates to each triple $(f, \fp, \sigma)$ the hash value $\sigma(\mathsf{h}_\cL(f; \fp))$. Combining \Cref{lemma:hashes,lemma:L-cover-S} we obtain a short list of pairs $f, g \in \mathcal{S}$ with the same hash values (for some choice of $\fp$, $\sigma$ and $\fq$, $\tau$). After comparing the remaining $a_\ell(f)$ and $a_\ell(g)$ for good primes $\ell \leq 1000$ we are left only with the examples given in \Cref{thm:congruence-classification}.

  The examples in \Cref{tab:mf-congs} are verified to be $p$-congruent by comparing the values $a_\ell(f)$ and $a_\ell(g)$ up to the bound from \Cref{prop:sturm-cong}\ref{iii:equalities}.
\end{proof}

\begin{remark}
  \label{rmk:big-p-L}
  \hfill
  \begin{enumerate}
  \item
    Note that in \Cref{lemma:L-cover-S} we could take instead a single set $\mathcal{L}$ consisting only of primes $\geq 10\,000$, after which $\mathcal{S}_\cL = \mathcal{S}$ (this is the approach taken by Cremona--Freitas~\cite{CF_GMFTSTOCBEC}). However, the $a_\ell(f)$ values used to compute $\mathsf{h}_{\mathcal{L}}(f ; \fp)$ are stored in the LMFDB for $\ell < 1000$ and we found it convenient to lookup, rather than recompute, these coefficients. 
  \item
    When $f$ has coefficient field $\QQ$ we have $\fp = (p)$ and there exists a unique embedding $\sigma \colon \FF_\fp \hookrightarrow \overbar{\FF}_p$. In particular, our hash value $\sigma(\mathsf{h}_{\cL}(f; \fp))$ collapses to that defined by Cremona--Freitas~\cite{CF_GMFTSTOCBEC}.
  \end{enumerate}
\end{remark}

\begin{table}[t]
  \centering
  \begin{tabular}{c|c}
    $i$ & $\cL_i$ \\
    \hline
    $0$ & $\{ 419, 431, 577, 587, 599, 617, 733, 773, 823, 859, 877, 883, 887, 941, 983  \}$ \\
    $1$ & $\{ 419, 439, 541, 569, 587, 641, 709, 727, 751, 769, 773, 821, 827, 859, 971  \}$ \\
    $2$ & $\{ 419, 461, 569, 577, 601, 641, 701, 719, 733, 751, 887, 907, 919, 971, 983  \}$ \\
    $3$ & $\{ 439, 461, 617, 631, 691, 733, 739, 757, 787, 811, 827, 919, 937, 947, 983  \}$ \\
    $4$ & $\{ 439, 503, 587, 647, 683, 701, 709, 727, 739, 757, 797, 829, 839, 853, 929  \}$ \\
    $5$ & $\{ 467, 541, 683, 691, 761, 773, 811, 829, 853, 863, 929, 937, 947 , 953, 991 \}$ \\
  \end{tabular}
  \caption{The sets of prime numbers $\cL_0, \dots, \cL_5$ from \Cref{lemma:L-cover-S}}
  \label{tab:Li}
\end{table}

\section{Filtering for visible \texorpdfstring{$\Sha$}{III} and \texorpdfstring{\Cref{thm:examples,thm:examples-vis}}{Theorems 1.1 and 1.4}}
\label{sec:filter}
Let $K$ be a number field and let $A/K$ and $B/K$ be abelian varieties. Let $\Delta/K$ be a finite group scheme and suppose that there exist closed immersions of group schemes $\Delta \hookrightarrow A$ and $\Delta \hookrightarrow B$ and let $Z = (A \times B)/\Delta$ (with $\Delta$ embedded diagonally). Consider the isogenies $\phi \colon A \to A'$ and $\psi \colon B \to B'$ whose kernel is $\Delta$.

\begin{theorem}[{\cite[Theorem~3.1]{AS_VOSTGOAV}, \cite[Theorem~2.2]{F_VEOO7ITTSGOAEC}}]
  \label{thm:vis-crit}
   If $d = |\Delta|$ is odd, $A'(K)/\phi(A(K))=0$ and: 
  \begin{enumerate}[label=(\roman*)]
  \item \label{i:vis-crit-tama}    
    the Tamagawa numbers of $A$ and $B'$ are coprime to $d$,
  \item \label{i:vis-crit-good}
    both $A$ and $B$ have good reduction at places $v \mid d$, and
  \item \label{i:vis-crit-ram}
    at finite places $v \mid d$ we have $e(K_v/\QQ_p) < p-1$.
  \end{enumerate}
  Then $\Vis_Z \Sha(A/K) \cong B'(K)/\psi(B(K))$.
\end{theorem}

\begin{remark}
  \label{rmk:B-prime}
  Note the presence of $B'$ in \Cref{thm:vis-crit}\ref{i:vis-crit-tama}. In \cite[Theorem~2.2]{F_VEOO7ITTSGOAEC} this is omitted (presumably this is a typographical error in \cite{F_VEOO7ITTSGOAEC}). A corrected proof is given in \Cref{sec:theorem-agashe-stein}. This does not play a role in \cite{F_VEOO7ITTSGOAEC} since all isogenies therein are endomorphisms.
\end{remark}

In order to apply \Cref{thm:vis-crit} to construct elements of $\Sha(A/\QQ)$ we require two main conditions:
\begin{enumerate}[label=(\arabic*)]
\item \label{i:cong-needed}
  there exists an abelian variety $B/\QQ$ with a finite subgroup scheme $\Delta \subset B$ embedding in $A$, and
\item \label{i:rank-needed}
  that $\rank(A/\QQ) = 0$ and $\rank(B/\QQ) \geq 1$,
\end{enumerate}
in addition to the local criteria \ref{i:vis-crit-tama}--\ref{i:vis-crit-good} (when $K = \QQ$ condition \ref{i:vis-crit-ram} is automatically satisfied).

Pairs of simple abelian varieties $A/\QQ$ and $B/\QQ$ satisfying \ref{i:cong-needed} are provided to us by \Cref{thm:congruence-classification-av}. In this case, $A$ and $B$ have real multiplication by maximal orders $\cO_A$ and $\cO_B$ in a pair of totally real fields $K_A$ and $K_B$, and $A$ and $B$ are $(\fp,\fq)$-congruent over $\QQ$ for some prime ideals $\fp \subset \cO_A$ and $\fq \subset \cO_B$, and have associated newforms $f \in S_2^{\textnormal{new}}(\Gamma_0(N))$ and $g \in S_2^{\textnormal{new}}(\Gamma_0(M))$ respectively (where $N,M \leq 10\,000$).
\\

\filter{good reduction} \label{fil:good}
In light of \Cref{thm:vis-crit}\ref{i:vis-crit-good} we discard pairs for which either $N$ or $M$ is divisible by the prime $p$ below $\fp$ and $\fq$.

\begin{remark}
  As demonstrated by Fisher in several examples in \cite[Theorem~3.3]{F_VEOO7ITTSGOAEC} it is sometimes possible to apply \Cref{thm:vis-crit} in cases when $p$ divides $N$ or $M$ by passing to a ramified extension of $\QQ_p$ over which $A$ and $B$ obtain good reduction. In particular, this filter may be finer than necessary. We do not investigate this further. 
\end{remark}

\filter{rank discrepancy} \label{fil:rk-disc}
Now, since $A$ is simple, its $L$-function decomposes as a product
\[
  L(A, s) = \prod_{\sigma \colon K_A \hookrightarrow \CC} L(f^\sigma, s)
\]
and similarly for $B$. In particular we have $\rank_{\textnormal{an}}(A) = [K_A : \QQ] \rank_{\textnormal{an}}(f)$ and $\rank_{\textnormal{an}}(B) = [K_B : \QQ] \rank_{\textnormal{an}}(g)$.

Each newform $f \in S_2^{\textnormal{new}}(\Gamma_0(N))$ of level $N \leq 10\,000$ and coefficient field of degree $\leq 2$ has its (numerical) analytic rank recorded in the LMFDB~\cite{LMFDB}. It is therefore simple to apply the filter that $\rank_{\textnormal{an}}(f) = 0$ and $\rank_{\textnormal{an}}(g) \geq 1$ to provide a shortlist of candidates for abelian varieties $A/\QQ$ with non-trivial $\Sha(A/\QQ)$.
\\

\filter{$\fq$ is a principal ideal} \label{fil:princ}
After applying Filters~\ref{fil:good} and \ref{fil:rk-disc} we can be confident (by the Birch and Swinnerton-Dyer conjecture) that if the Tamagawa number criterion \Cref{thm:vis-crit}\ref{i:vis-crit-tama} is satisfied, then $\Sha(A/\QQ)[\fp]$ is non-trivial. That is, we require that the Tamagawa numbers of $A$ and $B'$ are coprime to $p$.

In each case $A$ is presented to us as the Jacobian of a curve $C/\QQ$. In particular, we may (at least in principle) compute the Tamagawa numbers of $A$ by computing the minimal regular models for these curves at each bad prime.

On the other hand, the subgroup scheme $B[\fq] \subset B$ is not necessarily maximal isotropic with respect to the Weil pairing. In particular, although the abelian variety $B$ is presented to us as the Jacobian of a curve, the isogeny $\psi$ does not naturally equip $B'$ with a principal polarisation (and therefore cannot easily be presented as the Jacobian of a curve). In this case we do not know of a computationally feasible way to extract the Tamagawa numbers of $B'$.

However if $\fq = (\eta)$ is principal, then $\psi$ is the multiplication-by-$\eta$ endomorphism of $B$ and therefore $B' \cong B$. We therefore impose the further filter that $\fq$ is a principal ideal.
\\

\filter{Tamagawa numbers} \label{fil:tama}
The first three filters yield a shortlist of pairs of Jacobians $A/\QQ$ and $B/\QQ$ for which, to apply \Cref{thm:vis-crit}, it remains to check that the Tamagawa numbers of $A$ and $B$ are coprime to $p$.

If $J = A$ or $B$ is an elliptic curve we compute the Tamagawa numbers of $J$ using Tate's algorithm as implemented in \textsc{Magma}~\cite{MAGMA} and \textsc{Sagemath}~\cite{sagemath}.

We now describe the case when $J = A$ or $B$ has dimension $2$. We require that $J$ is known (by \cite{CEHJMPV_ADOCWMJ,CEHJMPV_ADOCWMJ_electronic}) to be the Jacobian of a genus $2$ curve $C/\QQ$. Let $\ell$ be a prime number and let $\Phi_{\ell}$ denote the component group of the Néron model of $J$ at $\ell$. We write $c_{\ell}(J) = | \Phi_{\ell}(\FF_\ell) |$ for the Tamagawa number and $\overbar{c}_{\ell}(J) = | \Phi_{\ell}(\overbar{\FF}_\ell) |$ for the order of the geometric component group. 

There exist several practical methods for determining the orders $\overbar{c}_{\ell}(J)$ (and often the Tamagawa numbers $c_\ell(J)$). Note that $c_\ell(J)$ divides $\overbar{c}_\ell(J)$, so it suffices for our application to show that $\overbar{c}_\ell(J)$ is coprime to $p$. 

\subsubsection*{General algorithms}
We first have the following general algorithms which apply to genus $2$ curves $C/\QQ$ (but which may not always return).
\begin{itemize}[itemsep=1mm,leftmargin=7mm]
\item
  Liu's program \texttt{genus2reduction} in \textsc{PARI/gp}~\cite{PARI} (which is available through \textsc{Sagemath}) computes $\overbar{c}_\ell(J)$.
\item
  Donnelly's \textsc{Magma} functions \texttt{RegularModel} and \texttt{ComponentGroup} to compute $\overbar{c}_\ell(J)$ (these were extended by van~Bommel~\cite{vB_ECOBSDIIG2} to compute $c_\ell(J)$).
\item
  Dokchitser--Jakovac's~\cite{DJ_RMOC} algorithm\footnote{The work \cite{DJ_RMOC} remains in preparation and we are grateful to Tim Dokchitser for performing these calculations for us. Cases in \Cref{tab:full-examples} which rely on these (unpublished) results are marked with a dagger $(\dagger)$ (these are excluded from \Cref{thm:examples}).} to compute $\overbar{c}_\ell(J)$.
\end{itemize}

When $\ell = 2$ we also mention recent work of Liu~\cite{L_DODCORS} which computes the minimal regular model of a hyperelliptic curve $C/\QQ_2$ and therefore yields $c_\ell(J)$ (according to \cite{L_DODCORS}, implementation is ongoing). If these programs fail to return an answer we also have the following approaches which can be applied depending on whether $\ell^2 \mid N$. 

\subsubsection*{The case \texorpdfstring{$\ell \parallel N$}{l||N}}
In this case $J$ has semistable (indeed, totally toric) reduction at $\ell$ (see \cite[Section~7.1]{BS_CGOPTQ}). In particular, we have access to the following additional methods.
\begin{itemize}[itemsep=1mm,leftmargin=7mm]
\item
  When $\ell$ is odd the Tamagawa number $c_\ell(J)$ can be computed from the cluster picture of $C$~(see \cite{M2D2}, \cite[Theorem~3.0.2]{B_VOTNOJOHCWSR}, \cite[Section~10]{hyperuser}) as implemented in \textsc{Sagemath}~\cite{clusters}.
\item
  If $\widetilde{J}$ is an optimal\footnote{A quotient $J_0(N) \to \widetilde{J}$ is said to be \emph{optimal} if the kernel is connected.} quotient of the modular Jacobian $J_0(N)$  an algorithm of Kohel--Stein~\cite[Section~3]{KS_CGOQOJ0} computes $\overbar{c}_\ell(\widetilde{J})$ from the corresponding newform (this is implemented in \textsc{Magma} as the function \texttt{ComponentGroupOrder}). If $\fp$ is principal and both $J[\fp]$ and $J[\overbar{\fp}]$ are irreducible then $p \mid \overbar{c}_\ell(\widetilde{J})$ if and only if $p \mid \overbar{c}_\ell(J)$, since there exists an isogeny $J \to \widetilde{J}$ of degree coprime to $p$ (all isogenies of degree divisible by $p$ factor as $\psi \circ \varphi$ where $\psi$ has degree coprime to $p$ and $\varphi$ is a composition of the endomorphisms $\fp$ and $\overbar{\fp}$).
\end{itemize}

When $\ell = 2$ work Gehrunger~\cite[Figure~10]{G_TCOTSMROG2CIRC2} (implemented in \textsc{Sagemath} \cite{Gehrunger_electronic}) computes the stable model of $C/\overbar{\mathbb{Q}}_2$ together with the ``thicknesses''\footnote{Note that by convention the ``thicknesses'' in \cite{G_TCOTSMROG2CIRC2} is normalised with respect to the valuation on $\overbar{\QQ}_2$ with $v(2) = 1$.} of the nodes. It may be possible to obtain the minimal regular model from this through repeated applications of \cite[Corollary~10.3.25]{L_AGAAC}.

\subsubsection*{The case \texorpdfstring{$\ell^2 \mid N$}{l{\textasciicircum}2 | N}}
In this case we either fall back on more general algorithms or use the following result when $p > 2\dim(J) + 1$.
\begin{itemize}[itemsep=1mm,leftmargin=7mm]
\item
  If $p > 2 \dim(J) + 1$ a theorem of Lenstra--Oort~\cite[Theorem~1.13]{LO_AVHPAR} implies (by the discussion in \cite[Section~3.7]{AS_VEFTBSDCFMAVOAR0}) that $p$ is coprime to $c_\ell(J)$ (cf. \cite{S_TNFOAVWPGR} for the case when $A$ has potentially good reduction).
\end{itemize}

\vspace{0.5em}
After applying the Filters \ref{fil:good}--\ref{fil:tama} we are left with the examples recorded in \Cref{tab:full-examples}. 

\begin{proof}[Proof of \Cref{thm:examples-vis}]
  Each of the abelian surfaces is checked to be geometrically simple by applying the criterion in \cite[Section~14.4]{CF_PTAMAOCOG2} and \cite{S_TS2DAVDOQWMWGORAL19}. Using \textsc{Magma} we check that $A(\QQ)[p] = B(\QQ)[p] = 0$. To prove \Cref{thm:examples-vis} it remains to check that in each case the (algebraic) ranks of $A$ and $B$ are correctly recorded in \Cref{tab:full-examples}. Upper bounds on the ranks are found via a $2$-descent (\texttt{RankBounds} in \textsc{Magma}) and lower bounds are found by a point search (\texttt{Points} in \textsc{Magma}). 
  In every case $\rank(B) = 2\dim(B)$ and therefore $B(\QQ)$ is an $\cO_B$-module of rank $2$. Since $\fq$ is principal and has residue degree $1$ we have $(\ZZ/p\ZZ)^2 \cong B(\QQ)/\fq B(\QQ)$. Since $A[\fp]$ is a direct summand of $A[p]$ we have $A'(\QQ)[p] = A(\QQ)[p] = 0$. Thus $A'(\QQ)/\fp A(\QQ) = 0$ because $A$ has rank $0$. It follows from \Cref{thm:vis-crit} that 
  \[
  (\ZZ/p\ZZ)^2 \cong B(\QQ)/\fq B(\QQ) \cong \Vis_Z \Sha(A/\QQ) 
  \]
  as required.
\end{proof}

\begin{proof}[Proof of \Cref{thm:examples}]
  This follows immediately from \Cref{thm:examples-vis} since the examples in \Cref{table:examples} consist only of those in \Cref{tab:full-examples} which are marked with neither a $(\dagger)$ or a $(*)$.
\end{proof}

\section{On \texorpdfstring{\Cref{conj:invisible-7}}{Conjecture 1.5}}
\label{sec:invis}
Let $A/\QQ$ be the abelian surface with label \LMFDBLabelMF{9603.2.a.o} and let $B/\QQ$ be the abelian surface with label \LMFDBLabelMF{9603.2.a.m}. By \Cref{lemma:have-RM} both $A$ and $B$ have real multiplication by $\ZZ[\sqrt{2}]$ defined over $\QQ$.

By \Cref{thm:congruence-classification-av} we may choose the embeddings $\ZZ[\sqrt{2}] \hookrightarrow \End_\QQ(A)$ and $\ZZ[\sqrt{2}] \hookrightarrow \End_\QQ(B)$ so that $A[\fp] \cong B[\fp]$ where $\fp = (3 + \sqrt{2})$. By \Cref{thm:examples-vis} the group $\Sha(A/\QQ)[\fp] \subset \Sha(A/\QQ)[7]$ contains a subgroup isomorphic to $(\ZZ/7\ZZ)^2$.
Note that, since $A[\fp]$ is irreducible, if an element $\xi \in \Sha(A/\QQ)[\fp]$ is visible in an abelian three-fold then there exists an elliptic curve $E/\QQ$ of rank at least $1$ which is $(7, \fp)$-congruent to $A$.

By \cite[Section~4.4]{PSS_TOX7APSTX2Y3Z7} (see \cite[Lemmas~7 and 8]{F_E7TITSGOG2J}) there exist twists $X_{A[\fp]}^+(7)$ and $X_{A[\fp]}^-(7)$ of the (connected) modular curve of full level $7$ whose $K$-points parametrise elliptic curves $E/K$ that are $(7,\fp)$-congruent to $A$ over $K$. Using the algorithms from \cite[Sections~3~and~4]{F_E7TITSGOG2J} we compute these twists explicitly as
\begin{flalign*}
    X_{A[\fp]}^{\pm}(7) :&\; 71 x^4 + 62 x^3 y - 42 x^3 z + 9 x^2 y^2 + 3 x^2 y z - 42 x^2 z^2 - 38 x y^3 + 39 x y^2 z &\\
    &+ 9 x y z^2 + 33 x z^3 - 6 y^4 + 25 y^3 z - 6 y^2 z^2 + 13 y z^3- 9 z^4, &\\
    X_{A[\fp]}^{\mp}(7) :&\; x^4 + 11 x^3 z + 15 x^2 y^2 + 21 x^2 y z + 9 x^2 z^2 + 21 x y^3 + 21 x y^2 z - 15 x y z^2 &\\
    &- 9 x z^3 + 6 y^4 + 41 y^3 z + 15 y^2 z^2 - 4 y z^3 - 15 z^4, &
\end{flalign*}
for some choice of sign.

Searching for rational points up to moderate height on these twisted modular curves yields the point $(-2:1:2) \in X_{A[\fp]}^\pm(7)(\QQ)$. We were unable to find any further points on either $X_{A[\fp]}^{\pm}(7)(\QQ)$ or $X_{A[\fp]}^{\mp}(7)(\QQ)$. The point above corresponds to the elliptic curve $(7,\fp)$-congruent to $A/\QQ$ 
\begin{align*}
E &: y^2 + y = x^3 - 28153533873 x + 1818290903847853.
\end{align*}

We check with a descent calculation (using the function \texttt{Rank} in \textsc{Magma}~\cite{MAGMA}) that $E(\QQ) = 0$. It follows immediately that the non-trivial elements of $\Sha(A/\QQ)[\fp]$ are not made visible by the congruence with $E$.
\Cref{conj:invisible-7} follows if one can prove the following conjecture.

\begin{conj}
    We have $X_{A[\fp]}^{\pm}(7)(\QQ) = \{ (-2 : 1 : 2) \}$ and $X_{A[\fp]}^{\mp}(7)(\QQ) = \emptyset$.
\end{conj}

It may be possible to prove that $X_{A[\fp]}^{\pm}(7)(\QQ) = \{ (-2 : 1 : 2) \}$ by applying the $2$-descent as in \cite[Section~11]{PSS_TOX7APSTX2Y3Z7} together with Chabauty's method. On the other hand we do not know of an approach to prove the non-existence of rational points on $X_{A[\fp]}^{\mp}(7)$ since there is no local obstruction (as is readily checked in \textsc{Magma}).

\appendix
\section{On the theorem of Agashe--Stein and Fisher}
\label{sec:theorem-agashe-stein}

The following proof of \Cref{thm:vis-crit} is almost identical to \cite[Theorem~3.1]{AS_VOSTGOAV} and \cite[Theorem~2.2]{F_VEOO7ITTSGOAEC} with the distinction that we allow an arbitrary finite subgroup $\Delta$ (whereas \cite{AS_VOSTGOAV} allows only the case $\Delta = B[n]$) and we correct a small typographical error from \cite{F_VEOO7ITTSGOAEC} (see \Cref{rmk:B-prime}).

\begin{proof}[Proof of \Cref{thm:vis-crit}]
Under the hypotheses of the theorem \cite[Lemma~2.1]{F_VEOO7ITTSGOAEC} implies that $\Vis_Z H^1(K, A) \cong B'(K)/\psi(B(K))$. It remains to show that the classes in this subgroup are contained in $\Sha(A/K)$. By \cite[(2.1)]{F_VEOO7ITTSGOAEC} for each place $v$ of $K$ we have the following commutative diagram with exact rows:
\[
\begin{tikzcd}
  & B(K_v) \arrow[r,"\psi"] \arrow[d] & B'(K_v) \arrow[r] \arrow[d,equal] & H^1(K_v,\Delta) \arrow[d]\\
  & Z(K_v) \arrow[r] & B'(K_v) \arrow[r,"\pi_v"] & H^1(K_v,A) \arrow[r,"{\iota}_*"] & H^1(K_v,Z).
\end{tikzcd}
\] 
We claim that $\pi_v = 0$ at each place $v$ of $K$. The claim suffices to prove the theorem, since by definition we have $\Vis_Z H^1(K, A) = \ker {\iota}_{*}$.

Recall that we write $d = | \Delta |$. Since $\pi_v$ factors through the group $H^1(K_v, \Delta)$ (which is annihilated by multiplication by $d$) it follows that $d \cdot \pi_v = 0$.

If $v$ is an infinite place then $H^1(K_v,A)$ is killed by multiplication by $2$, but by assumption $d$ is odd, and the claim follows. Thus, suppose that $v$ is a finite place of $K$. Let $K_v^{\text{nr}}$ be the maximal unramified extension of $K_v$. For an abelian variety $J/K_v$ we write $J^0(K_v^{\text{nr}}) \subset J(K_v^{\text{nr}})$ for the set of points whose reduction modulo $v$ lie in the identity component of the special fibre of the Néron model of $J/K_v$. We now recall the following fact from \cite[Section~3.4]{AS_VOSTGOAV}.

\begin{lemma}
  \label{lemma:what-are-the-cp}
  For any finite place $v$ of $K$ the Tamagawa number $c_v(A)$ is equal to the order of the unramified subgroup of $H^1(K_v, A)$.
\end{lemma}

First if $v \mid d$ the argument in \cite[Section~3.4, Case 3]{AS_VOSTGOAV} implies that (under hypotheses \ref{i:vis-crit-good} and \ref{i:vis-crit-ram} of \Cref{thm:vis-crit}) the map $Z(K_v^{\text{nr}}) \to B'(K_v^{\text{nr}})$ is surjective. But then the image of $\pi_v$ is contained in the unramified subgroup of $H^1(K_v, A)$, and the claim follows from \Cref{lemma:what-are-the-cp} and the assumption that $d$ is coprime to $c_v(A)$.

Now suppose that $v \nmid d$. In this case it is shown in \cite[Section~3.4, Case 2]{AS_VOSTGOAV} that the map $\psi : B^0(K_v^{\text{nr}}) \to B'^0(K_v^{\text{nr}})$ is surjective.

For $P \in B'(K_v)$ we have $c_{v}(B') \cdot P \in B'^0(K_v)$ since $c_v(B')$ is the order of the component group of the Néron model of $B'$ at $v$. Therefore there exists $Q \in B^0(K_v^{\text{nr}})$ such that $\psi(Q) = c_v(B') \cdot P$.

Thus $c_v(B') \cdot \pi_v(P) = \pi_v( c_v(B') \cdot P)$ is in the same cohomology class as the cocycle $\sigma \mapsto Q^{\sigma}-Q$, and in particular $c_v(B') \cdot \pi_v(P)$ is unramified. But by \Cref{lemma:what-are-the-cp} this implies that 
\begin{equation*}
    c_v(A) c_v(B') \cdot \pi_v(P) = 0 .
\end{equation*}
Combining this with the fact that $d \cdot \pi_v = 0$ and the hypothesis that $c_v(A) c_v(B')$ is coprime to $d$ we deduce $\pi_v(P) = 0$, as required.
\end{proof}

\section{Tables}
\label{sec:tables}

\subsection*{How to read the tables}
\label{sec:how-read-tables}

\subsubsection*{\Cref{tab:mf-congs}}
We list all weight $2$ newforms $f \in S_2^{\textnormal{new}}(\Gamma_0(N))$ and $g \in S_2^{\textnormal{new}}(\Gamma_0(M))$ in the LMFDB  (i.e., $N,M \leq 10\,000$) with coefficient fields of degree $\leq 2$ which are $p$-congruent modulo an unramified prime where $p \geq 23$. Note the existence of the triple $\LMFDBLabelMF{4725.2.a.bb}, \LMFDBLabelMF{4725.2.a.bc}, \LMFDBLabelMF{4725.2.a.x}$ of $29$-congruent forms. The corresponding (isomorphic) mod $p$ Galois representations are irreducible in every example in \Cref{tab:mf-congs}.

\subsubsection*{\Cref{tab:full-examples}}
We list certain examples of abelian varieties $A/\QQ$ and $B/\QQ$ with real multiplication by an order in the respective fields $K_A$ (resp. $K_B$) such that a subgroup $(\ZZ/p\ZZ)^2 \subset \Sha(A/\QQ)[\fp]$ is made visible by a $(\fp,\fq)$-congruence between $A$ and $B$ for some prime ideals $\fp \subset \cO_A$ and $\fq \subset \cO_B$. For each $J = A$ or $B$ we record:
\begin{itemize}[itemsep=1mm,leftmargin=7mm]
    \item The label of the (Galois orbit of) weight $2$ newforms $f_J \in S_2^{\textnormal{new}}(\Gamma_0(N))$ associated\footnote{This association is conditional if the corresponding entry in \Cref{tab:weier-eqns} is marked with a $(\star)$, see \Cref{lemma:correct-eqns-forms}. In every case $A$ and $B$ are $(\fp,\fq)$-congruent by \Cref{thm:congruence-classification-av}.} to the isogeny class of $J$.
    \begin{itemize}
        \item If $\deg(K_J) = 1$ the elliptic curve $J$ is chosen to be $\Gamma_0(N)$-optimal, and
        \item If $\deg(K_J) = 2$ the abelian surface $J$ is chosen to be the Jacobian of the corresponding genus $2$ curve listed in \Cref{tab:weier-eqns}. 
    \end{itemize}
    \item The (algebraic) rank of $J/\QQ$.
    \item An integer $\mathfrak{c}(J)$ such that the product of Tamagawa numbers $\prod_\ell c_\ell(J)$ divides $\mathfrak{c}(J)$ (when $J$ is an elliptic curve $\mathfrak{c}(J) = c(J) = \prod_\ell c_\ell(J)$). The term $\mathcal{P}$ denotes the product $\mathcal{P} = \prod_{\ell^2 \mid N} c_\ell(J)$ (by \cite[Theorem~1.13]{LO_AVHPAR} $\mathcal{P}$ is supported on primes $q$ for which $q^2 \mid N$ or $q \leq 2 \dim(J) + 1$). 
\end{itemize}
 
Cases where the Tamagawa product was computed using the unpublished work \cite{DJ_RMOC} are marked with a dagger $(\dagger)$. Cases in which we could not compute a Tamagawa number (and are thus conditional) are marked with an asterisk $(*)$. In the example \LMFDBLabelMF{9802.2.a.j} we failed to compute the order of the component group at $2$. However, since $2 \parallel 9802$ the algorithm of Kohel--Stein~\cite{KS_CGOQOJ0} (the \textsc{Magma} function \texttt{ComponentGroupOrder}) can be applied in this case. Unfortunately, we were unable to obtain a result after several hours.

\subsubsection*{\Cref{tab:weier-eqns}}
We give explicit Weierstrass equations for the genus $2$ curves whose Jacobians have a $p$-torsion element in $\Sha(J/\QQ)$ by \Cref{thm:examples}. Note that this relies heavily on the work \cite{CEHJMPV_ADOCWMJ,CEHJMPV_ADOCWMJ_electronic} in preparation where the authors construct RM abelian whose associated newforms can be found in the LMFDB. We prove in \Cref{lemma:correct-eqns-forms} that the (isogeny class of) abelian varieties corresponds to the newforms, those cases which we could not prove rigorously are marked with a star $(\star)$.
\FloatBarrier 

\vspace*{\fill}
\begingroup
\small
\begin{table}[H]
  \centering
  \begin{tabular}{c|cc||c|cc}
    $p$ & Label of $f$                        & Label of $g$                        &  $p$ & Label of $f$                        & Label of $g$                        \\
    \hline
    23  & \LMFDBLabelMF{1755.2.a.g}  & \LMFDBLabelMF{1755.2.a.l}  & 23   & \LMFDBLabelMF{8820.2.a.bd} & \LMFDBLabelMF{8820.2.a.bg} \\
    23  & \LMFDBLabelMF{1755.2.a.h}  & \LMFDBLabelMF{1755.2.a.k}  & 23   & \LMFDBLabelMF{8820.2.a.bi} & \LMFDBLabelMF{8820.2.a.bl} \\
    23  & \LMFDBLabelMF{2205.2.a.ba} & \LMFDBLabelMF{2205.2.a.q}  & 29   & \LMFDBLabelMF{2178.2.a.bb} & \LMFDBLabelMF{2178.2.a.x}  \\
    23  & \LMFDBLabelMF{2205.2.a.n}  & \LMFDBLabelMF{2205.2.a.z}  & 29   & \LMFDBLabelMF{2178.2.a.p}  & \LMFDBLabelMF{2178.2.a.t}  \\
    23  & \LMFDBLabelMF{3332.2.a.h}  & \LMFDBLabelMF{3332.2.a.k}  & 29   & \LMFDBLabelMF{2250.2.a.c}  & \LMFDBLabelMF{2250.2.a.f}  \\
    23  & \LMFDBLabelMF{4510.2.a.o}  & \LMFDBLabelMF{4510.2.a.q}  & 29   & \LMFDBLabelMF{2250.2.a.m}  & \LMFDBLabelMF{2250.2.a.n}  \\
    23  & \LMFDBLabelMF{5320.2.a.i}  & \LMFDBLabelMF{5320.2.a.l}  & 29   & \LMFDBLabelMF{2394.2.a.w}  & \LMFDBLabelMF{2394.2.a.z}  \\
    23  & \LMFDBLabelMF{6358.2.a.a}  & \LMFDBLabelMF{6358.2.a.k}  & 29   & \LMFDBLabelMF{3249.2.a.i}  & \LMFDBLabelMF{3249.2.a.j}  \\
    23  & \LMFDBLabelMF{6358.2.a.g}  & \LMFDBLabelMF{6358.2.a.j}  & 29   & \LMFDBLabelMF{3249.2.a.o}  & \LMFDBLabelMF{3249.2.a.p}  \\
    23  & \LMFDBLabelMF{6650.2.a.bu} & \LMFDBLabelMF{6650.2.a.bx} & 29   & \LMFDBLabelMF{4356.2.a.r}  & \LMFDBLabelMF{4356.2.a.t}  \\
    23  & \LMFDBLabelMF{6942.2.a.q}  & \LMFDBLabelMF{6942.2.a.s}  & 29   & \LMFDBLabelMF{4356.2.a.u}  & \LMFDBLabelMF{4356.2.a.w}  \\
    23  & \LMFDBLabelMF{7050.2.a.bu} & \LMFDBLabelMF{7050.2.a.bv} & 29   & \LMFDBLabelMF{4725.2.a.bb} & \LMFDBLabelMF{4725.2.a.bc} \\
    23  & \LMFDBLabelMF{7425.2.a.bd} & \LMFDBLabelMF{7425.2.a.bl} & 29   & \LMFDBLabelMF{4725.2.a.bb} & \LMFDBLabelMF{4725.2.a.x}  \\
    23  & \LMFDBLabelMF{7425.2.a.bi} & \LMFDBLabelMF{7425.2.a.z}  & 29   & \LMFDBLabelMF{5418.2.a.bb} & \LMFDBLabelMF{5418.2.a.x}  \\
    23  & \LMFDBLabelMF{7560.2.a.bb} & \LMFDBLabelMF{7560.2.a.x}  & 29   & \LMFDBLabelMF{5742.2.a.bk} & \LMFDBLabelMF{5742.2.a.bl} \\
    23  & \LMFDBLabelMF{7560.2.a.bg} & \LMFDBLabelMF{7560.2.a.bj} & 29   & \LMFDBLabelMF{6426.2.a.be} & \LMFDBLabelMF{6426.2.a.bm} \\
    23  & \LMFDBLabelMF{7650.2.a.c}  & \LMFDBLabelMF{7650.2.a.cw} & 29   & \LMFDBLabelMF{6426.2.a.bn} & \LMFDBLabelMF{6426.2.a.bv} \\
    23  & \LMFDBLabelMF{7650.2.a.cj} & \LMFDBLabelMF{7650.2.a.dc} & 29   & \LMFDBLabelMF{6890.2.a.q}  & \LMFDBLabelMF{6890.2.a.r}  \\
    23  & \LMFDBLabelMF{8085.2.a.ba} & \LMFDBLabelMF{8085.2.a.bh} & 31   & \LMFDBLabelMF{6650.2.a.bq} & \LMFDBLabelMF{6650.2.a.bv} \\
    23  & \LMFDBLabelMF{8775.2.a.be} & \LMFDBLabelMF{8775.2.a.t}  & 59   & \LMFDBLabelMF{4332.2.a.e}  & \LMFDBLabelMF{4332.2.a.f}  \\
    23  & \LMFDBLabelMF{8775.2.a.bf} & \LMFDBLabelMF{8775.2.a.s}  & 59   & \LMFDBLabelMF{4332.2.a.j}  & \LMFDBLabelMF{4332.2.a.k}  \\
  \end{tabular}
  \caption{All pairs $(f,g)$ of weight $2$ newforms in the LMFDB whose coefficient fields have degree $\leq 2$ and which are $p$-congruent modulo an unramified prime where $23 \leq p \leq \bigp$. For every form in this table the corresponding mod $p$ Galois representation is irreducible.}
  \label{tab:mf-congs}
\end{table}
\endgroup
\vspace*{\fill}
\clearpage

\FloatBarrier
\begingroup
\small
\begin{table}[p]
  \centering
\begin{tabular}{lll|ccc|ccc}
    $p$          & Label of $f_A$                        & Label of $f_B$                       & $K_A$            & $\rank(A)$ & $\mathfrak{c}(A)$       & $K_B$           & $\rank(B)$ & $\mathfrak{c}(B)$      \\
  \hline
  $5$          & \LMFDBLabelMF{1058.2.a.e}  & \LMFDBLabelMF{1058.2.a.a} & $\QQ$            & $0$        & $1$                     & $\QQ$           & $2$        & $2$                    \\
  $5$          & \LMFDBLabelMF{1246.2.a.g}  & \LMFDBLabelMF{1246.2.a.a} & $\QQ$            & $0$        & $4$                     & $\QQ$           & $2$        & $4$                    \\
  $5$          & \LMFDBLabelMF{1664.2.a.r}  & \LMFDBLabelMF{1664.2.a.a} & $\QQ$            & $0$        & $2$                     & $\QQ$           & $2$        & $4$                    \\
  $5$          & \LMFDBLabelMF{2366.2.a.g}  & \LMFDBLabelMF{2366.2.a.a} & $\QQ$            & $0$        & $1$                     & $\QQ$           & $2$        & $12$                   \\
  $5$          & \LMFDBLabelMF{2574.2.a.j}  & \LMFDBLabelMF{2574.2.a.a} & $\QQ$            & $0$        & $2$                     & $\QQ$           & $2$        & $8$                    \\
  $5$  & \LMFDBLabelMF{2736.2.a.bb}$(\dagger)$ & \LMFDBLabelMF{2736.2.a.a} & $\QQ(\sqrt{41})$ & $0$        & $256$                   & $\QQ$           & $2$        & $16$                   \\
  $5$          & \LMFDBLabelMF{2834.2.a.e}  & \LMFDBLabelMF{2834.2.a.c} & $\QQ$            & $0$        & $1$                     & $\QQ$           & $2$        & $24$                   \\
  $5$          & \LMFDBLabelMF{3384.2.a.f}  & \LMFDBLabelMF{3384.2.a.a} & $\QQ$            & $0$        & $8$                     & $\QQ$           & $2$        & $16$                   \\
  $5$          & \LMFDBLabelMF{3952.2.a.g}  & \LMFDBLabelMF{3952.2.a.d} & $\QQ$            & $0$        & $1$                     & $\QQ$           & $2$        & $4$                    \\
  $5$          & \LMFDBLabelMF{4092.2.a.d}  & \LMFDBLabelMF{4092.2.a.a} & $\QQ$            & $0$        & $4$                     & $\QQ$           & $2$        & $12$                   \\
  $5$          & \LMFDBLabelMF{4592.2.a.k}  & \LMFDBLabelMF{4592.2.a.a} & $\QQ$            & $0$        & $8$                     & $\QQ$           & $2$        & $4$                    \\
  $5$          & \LMFDBLabelMF{4968.2.a.bf} & \LMFDBLabelMF{4968.2.a.a} & $\QQ(\sqrt{41})$ & $0$        & $72$                    & $\QQ$           & $2$        & $12$                   \\
  $5$          & \LMFDBLabelMF{5056.2.a.ba} & \LMFDBLabelMF{5056.2.a.d} & $\QQ(\sqrt{6})$  & $0$        & $32$                    & $\QQ$           & $2$        & $4$                    \\
  $5$          & \LMFDBLabelMF{5082.2.a.h}  & \LMFDBLabelMF{5082.2.a.c} & $\QQ$            & $0$        & $1$                     & $\QQ$           & $2$        & $24$                   \\
  $5$  & \LMFDBLabelMF{5221.2.a.b}$(\dagger)$  & \LMFDBLabelMF{5221.2.a.a} & $\QQ(\sqrt{29})$ & $0$        & $7$                     & $\QQ$           & $2$        & $1$                    \\
  $5$          & \LMFDBLabelMF{5586.2.a.d}  & \LMFDBLabelMF{5586.2.a.c} & $\QQ$            & $0$        & $2$                     & $\QQ$           & $2$        & $4$                    \\
  $5$          & \LMFDBLabelMF{5904.2.a.o}  & \LMFDBLabelMF{5904.2.a.b} & $\QQ$            & $0$        & $8$                     & $\QQ$           & $2$        & $16$                   \\
  $5$      & \LMFDBLabelMF{6864.2.a.bf}$(*)$ & \LMFDBLabelMF{6864.2.a.a} & $\QQ(\sqrt{41})$ & $0$        & $\boldsymbol{?}$        & $\QQ$           & $2$        & $8$                    \\
  $5$          & \LMFDBLabelMF{7014.2.a.k}  & \LMFDBLabelMF{7014.2.a.g} & $\QQ$            & $0$        & $1$                     & $\QQ$           & $2$        & $24$                   \\
  $5$          & \LMFDBLabelMF{7084.2.a.j}  & \LMFDBLabelMF{7084.2.a.a} & $\QQ$            & $0$        & $1$                     & $\QQ$           & $2$        & $24$                   \\
  $5$      & \LMFDBLabelMF{7632.2.a.v}$(*)$  & \LMFDBLabelMF{7632.2.a.h} & $\QQ(\sqrt{41})$ & $0$        & $\boldsymbol{?}$        & $\QQ$           & $2$        & $16$                   \\
  $5$          & \LMFDBLabelMF{8526.2.a.f}  & \LMFDBLabelMF{8526.2.a.a} & $\QQ$            & $0$        & $1$                     & $\QQ$           & $2$        & $8$                    \\
  $5$          & \LMFDBLabelMF{8568.2.a.j}  & \LMFDBLabelMF{8568.2.a.a} & $\QQ$            & $0$        & $8$                     & $\QQ$           & $2$        & $32$                   \\
  $5$          & \LMFDBLabelMF{8784.2.a.bj} & \LMFDBLabelMF{8784.2.a.a} & $\QQ(\sqrt{6})$  & $0$        & $384$                   & $\QQ$           & $2$        & $16$                   \\
  $5$          & \LMFDBLabelMF{8866.2.a.g}  & \LMFDBLabelMF{8866.2.a.b} & $\QQ$            & $0$        & $4$                     & $\QQ$           & $2$        & $6$                    \\
  $5$      & \LMFDBLabelMF{9802.2.a.j}$(*)$  & \LMFDBLabelMF{9802.2.a.b} & $\QQ(\sqrt{29})$ & $0$        & $\boldsymbol{?}$        & $\QQ$           & $2$        & $4$                    \\
  $5$          & \LMFDBLabelMF{9936.2.a.bp} & \LMFDBLabelMF{9936.2.a.d} & $\QQ$            & $0$        & $2$                     & $\QQ$           & $2$        & $4$                    \\
  $7$          & \LMFDBLabelMF{3200.2.a.bm} & \LMFDBLabelMF{3200.2.a.a} & $\QQ(\sqrt{2})$  & $0$        & $\mathcal{P}$           & $\QQ$           & $2$        & $12$                   \\
  $7$          & \LMFDBLabelMF{5220.2.a.r}  & \LMFDBLabelMF{5220.2.a.a} & $\QQ(\sqrt{57})$ & $0$        & $\mathcal{P} \cdot 6$   & $\QQ$           & $2$        & $24$                   \\
  $7$          & \LMFDBLabelMF{8464.2.a.bg} & \LMFDBLabelMF{8464.2.a.d} & $\QQ(\sqrt{2})$  & $0$        & $\mathcal{P}$           & $\QQ$           & $2$        & $6$                    \\
  $7$          & \LMFDBLabelMF{8475.2.a.r}  & \LMFDBLabelMF{8475.2.a.a} & $\QQ(\sqrt{2})$  & $0$        & $\mathcal{P}$           & $\QQ$           & $2$        & $36$                   \\
  $7$          & \LMFDBLabelMF{8976.2.a.bm} & \LMFDBLabelMF{8976.2.a.a} & $\QQ(\sqrt{2})$  & $0$        & $\mathcal{P} \cdot 146$ & $\QQ$           & $2$        & $4$                    \\
  $7$  & \LMFDBLabelMF{9510.2.a.f}  & \LMFDBLabelMF{9510.2.a.x}$(\dagger)$ & $\QQ$         & $0$        & $6$                     & $\QQ(\sqrt{2})$ & $4$        & $544$                  \\
  $7$          & \LMFDBLabelMF{9603.2.a.o}  & \LMFDBLabelMF{9603.2.a.m} & $\QQ(\sqrt{2})$  & $0$        & $\mathcal{P} \cdot 71$  & $\QQ(\sqrt{2})$ & $4$        & $\mathcal{P} \cdot 2$  \\
  $11$ & \LMFDBLabelMF{6962.2.a.q}  & \LMFDBLabelMF{6962.2.a.o}$(\dagger)$ & $\QQ(\sqrt{5})$  & $0$        & $\mathcal{P} \cdot 41$  & $\QQ(\sqrt{5})$ & $4$        & $\mathcal{P} \cdot 76$ \\
  $11$         & \LMFDBLabelMF{9025.2.a.r}  & \LMFDBLabelMF{9025.2.a.n} & $\QQ(\sqrt{5})$  & $0$        & $\mathcal{P} \cdot 1$   & $\QQ(\sqrt{5})$ & $4$        & $\mathcal{P} \cdot 1$  \\
  $13$         & \LMFDBLabelMF{6776.2.a.r}  & \LMFDBLabelMF{6776.2.a.b} & $\QQ(\sqrt{3})$  & $0$        & $\mathcal{P} \cdot 3$   & $\QQ$           & $2$        & $32$                   \\
\end{tabular}
  \caption{Certain abelian varieties $A/\QQ$ and $B/\QQ$ with real multiplication by an order in the field $K_A$ (resp. $K_B$) and such that a subgroup $(\ZZ/p\ZZ)^2 \subset \Sha(A/\QQ)[\fp]$ is made visible by a $(\fp,\fq)$-congruence between $A$ and $B$. For each $J = A$ or $B$ we record the label of $f_J \in S_2^{\textnormal{new}}(\Gamma_0(N))$. If $\deg(K_J) = 1$ then $J$ is chosen to be $\Gamma_0(N)$-optimal and if $\deg(K_J) = 2$ then $J$ is the Jacobian of the corresponding genus $2$ curve listed in \Cref{tab:weier-eqns}.} \label{tab:full-examples}
\end{table}
\endgroup

\begingroup
\small
\clearpage
\begin{landscape}
  \pagestyle{plain}
  \begin{table}
    \centering
    \begin{tabular}{l|c|c}
      Label of $f$ &  $K_f$                   & Weierstrass equation                                                                                                                     \\
      \hline
      \LMFDBLabelMF{2736.2.a.bb} & $\QQ(\sqrt{41})$ & $y^2 = -168 x^{6} + 828 x^{5} - 282 x^{4} + 633 x^{3} - 3189 x^{2} - 1629 x - 2997$                                                      \\
      \LMFDBLabelMF{3200.2.a.bm} & $\QQ(\sqrt{2})$ & $y^2 = -520 x^{6} + 680 x^{5} - 60 x^{4} - 240 x^{3} + 80 x^{2} + 20 x - 10$                                                             \\
      \LMFDBLabelMF{4968.2.a.bf} & $\QQ(\sqrt{41})$ & $y^2 = -408 x^{6} + 3690 x^{5} - 15084 x^{4} + 34470 x^{3} - 45351 x^{2} + 31176 x - 7368$                                               \\
      \LMFDBLabelMF{5056.2.a.ba} & $\QQ(\sqrt{6})$ & $y^2 = 3608 x^{6} - 47912 x^{5} + 225314 x^{4} - 473684 x^{3} + 419938 x^{2} - 97768 x + 6504$                                           \\
      \LMFDBLabelMF{5220.2.a.r}  & $\QQ(\sqrt{57})$ & $y^2 = 34992 x^{5} + 1939380 x^{4} + 42861081 x^{3} + 472182963 x^{2} + 2593258155 x + 5680766145$                                       \\
      \LMFDBLabelMF{5221.2.a.b}  & $\QQ(\sqrt{29})$ & $y^2 = 16 x^{5} + 1576 x^{4} + 58697 x^{3} + 1072475 x^{2} + 10723808 x + 45456220$                                                      \\
      \LMFDBLabelMF{6776.2.a.r}  & $\QQ(\sqrt{3})$ & $y^2 = -100936 x^{6} - 375441 x^{5} + 574585 x^{4} + 2425533 x^{3} - 3932753 x^{2} + 1298528 x - 219604$                                 \\
      \LMFDBLabelMF{6864.2.a.bf} & $\QQ(\sqrt{41})$ & $y^2 = 16 x^{5} + 504 x^{4} -10815 x^{3} -115171 x^{2} + 1444344x^{1} -14672700$                                                         \\
      \LMFDBLabelMF{6962.2.a.o}$(\star)$  & $\QQ(\sqrt{5})$ & $y^2 - (x^{2} + x) y = -16x^{5} - 21x^{4} + 34x^{3} + 155x^{2} + 176x + 64$                                                              \\
      \LMFDBLabelMF{6962.2.a.q}$(\star)$  & $\QQ(\sqrt{5})$ & $y^2 = -236 x^{6} - 1652 x^{5} - 295 x^{4} + 10502 x^{3} - 4071 x^{2} - 4484 x - 708$                                                    \\
      \LMFDBLabelMF{7632.2.a.v} & $\QQ(\sqrt{41})$ & $y^2 = -1295760 x^{6} - 16237368 x^{5} - 82883841 x^{4} - 222068034 x^{3} - 330803205 x^{2} - 260544348 x - 84941220$                    \\
      \LMFDBLabelMF{8464.2.a.bg} & $\QQ(\sqrt{2})$ & $y^2 = -5152 x^{6} + 47012 x^{5} - 166313 x^{4} + 259532 x^{3} - 116403 x^{2} - 113344 x + 99659$                                        \\
      \LMFDBLabelMF{8475.2.a.r}  & $\QQ(\sqrt{2})$ & $y^2 - (x^{3} + 1)y = 3056991 x^{6} + 46889085 x^{5} + 299096355 x^{4} + 1015447247 x^{3} + 1934906580 x^{2} + 1961596845 x + 826414211$ \\
      \LMFDBLabelMF{8784.2.a.bj} & $\QQ(\sqrt{6})$ & $y^2 = -36 x^{5} + 15 x^{4} - 231 x^{3} + 177 x^{2} + 81 x - 63$                                                                         \\
      \LMFDBLabelMF{8976.2.a.bm} & $\QQ(\sqrt{2})$ & $y^2 = -207558 x^{6} - 394134 x^{5} - 320476 x^{4} - 182650 x^{3} - 75478 x^{2} - 14976 x - 1053$                                        \\
      \LMFDBLabelMF{9025.2.a.n}  & $\QQ(\sqrt{5})$ & $y^2 - (x^{3} + x^{2} + 1)y = 756x^{6} - 2583x^{5} + 3676x^{4} - 2773x^{3} + 1157x^{2} - 250x + 21$                                      \\
      \LMFDBLabelMF{9025.2.a.r}  & $\QQ(\sqrt{5})$ & $y^2 - (x^{2} + x)y = 95 x^{6} + 166 x^{4} + 617 x^{3} + 641 x^{2} + 760 x + 285$                                                        \\
      \LMFDBLabelMF{9510.2.a.x}$(\star)$ & $\QQ(\sqrt{2})$  & $y^2 - (x + 1)y = 3600x^{6} - 10200x^{5} + 11189x^{4} - 5864x^{3} + 1447x^{2} - 143x + 6$                                                \\
      \LMFDBLabelMF{9603.2.a.m}  & $\QQ(\sqrt{2})$ & $y^2 - (x^3 + 1)y = 2210x^{6} - 6246x^{5} + 7230x^{4} - 4406x^{3} + 1497x^{2} - 270x + 20$                                               \\
      \LMFDBLabelMF{9603.2.a.o}  & $\QQ(\sqrt{2})$ & $y^2 = -15 x^{6} + 66 x^{5} - 1113 x^{4} + 7086 x^{3} - 26139 x^{2} + 91284 x - 151440$                                                  \\
      \LMFDBLabelMF{9802.2.a.j}  & $\QQ(\sqrt{29})$ & $y^2 = 2704x^{5} - 1352x^{4} - 22399x^{3} - 10582x^{2} + 69017x + 52468$                                                                 \\
    \end{tabular}
    \caption{Weierstrass equations of curves whose Jacobians $J/\QQ$ have associated (Galois orbits of) newforms $f \in S_2^{\textnormal{new}}(\Gamma_0(N))$ with LMFDB label given by the first column. In those cases marked with a $(\star)$ we were unable to rigorously verify that $f$ and $J$ correspond, as discussed in \Cref{lemma:correct-eqns-forms}.}
    \label{tab:weier-eqns}
  \end{table}
\end{landscape}
\endgroup

\FloatBarrier
\newcommand{\etalchar}[1]{$^{#1}$}
\providecommand{\bysame}{\leavevmode\hbox to3em{\hrulefill}\thinspace}
\providecommand{\MR}{\relax\ifhmode\unskip\space\fi MR }
\providecommand{\MRhref}[2]{%
  \href{http://www.ams.org/mathscinet-getitem?mr=#1}{#2}
}
\providecommand{\bibtitleref}[2]{%
  \hypersetup{urlbordercolor=0.8 1 1}%
  \href{#1}{#2}%
  \hypersetup{urlbordercolor=cyan}%
}
\providecommand{\href}[2]{#2}

\end{document}